\documentclass[12pt, a4paper, reqno]{amsart}
\usepackage{amssymb, mathrsfs, tikz}
\usepackage[margin=2.0cm]{geometry}
\usepackage[pdftitle={KMS states on noncommutative solenoids},
  pdfauthor={Nathan Brownlowe, Mitch Hawkins and Aidan Sims},
  pdfsubject={Mathematics}]{hyperref}

\def\Aut{\operatorname{Aut}}
\def\lsp{\operatorname{span}}
\def\clsp{\overline{\operatorname{span}}}
\def\id{\operatorname{id}}
\def\sub{\operatorname{sub}}

\def\C{\mathbb{C}}
\def\R{\mathbb{R}}
\def\N{\mathbb{N}}
\def\Z{\mathbb{Z}}
\def\T{\mathbb{T}}

\def\AA{\mathcal{A}}
\def\KK{\mathcal{K}}
\def\LL{\mathcal{L}}
\def\OO{\mathcal{O}}
\def\TT{\mathcal{T}}

\newtheorem{thm}{Theorem}[section]
\newtheorem{cor}[thm]{Corollary}
\newtheorem{lemma}[thm]{Lemma}
\newtheorem{proposition}[thm]{Proposition}

\theoremstyle{definition}
\newtheorem{definition}[thm]{Definition}
\newtheorem{notation}[thm]{Notation}
\theoremstyle{remark}
\newtheorem{remark}[thm]{Remark}
\newtheorem{remarks}[thm]{Remarks}

\numberwithin{equation}{section}

\newcommand*{\mdis}[2]{M_{\sub}(#1,#2)}
\newcommand*{\mcon}[1]{\Omega_{\sub}^{#1}}

\newcommand{\Mod}[1]{\ (\text{mod}\ #1)}

\newcommand{\SI}{\mathbb{S}}

\newcommand{\KMS}{\ensuremath{\operatorname{KMS}}}

\begin{document}

\date{\today}
\title[The Toeplitz noncommutative solenoid and its KMS states]{The Toeplitz
noncommutative solenoid and its KMS states}

\author[Brownlowe]{Nathan Brownlowe}
\address{Nathan Brownlowe\\
    School of Mathematics and Statistics\\
    University of Sydney\\
    NSW  2006\\
    AUSTRALIA}
\email{nathan.brownlowe@sydney.edu.au}

\author[Hawkins]{Mitchell Hawkins}
\author[Sims]{Aidan Sims}
\address{Mitchell Hawkins and Aidan Sims\\
    School of Mathematics and
	Applied Statistics\\
	University of Wollongong\\
	NSW 2522\\
	AUSTRALIA}
\email{mrhawkins1989@gmail.com, asims@uow.edu.au}

\subjclass{46L55 (primary); 28D15, 37A55 (secondary)}

\thanks{This research was supported by the Australian Research Council grant DP150101595}

\begin{abstract}
We use Katsura's topological graphs to define Toeplitz extensions of Latr\'emoli\`ere and
Packer's noncommutative-solenoid $C^*$-algebras. We identify a natural dynamics on each
Toeplitz noncommutative solenoid and study the associated KMS states. Our main result
shows that the space of extreme points of the KMS simplex of the Toeplitz noncommutative
torus at a strictly positive inverse temperature is homeomorphic to a solenoid; indeed,
there is an action of the solenoid group on the Toeplitz noncommutative solenoid that
induces a free and transitive action on the extreme boundary of the KMS simplex. With the
exception of the degenerate case of trivial rotations, at inverse temperature zero there
is a unique KMS state, and only this one factors through Latr\'emoli\`ere and Packer's
noncommutative solenoid.
\end{abstract}

\maketitle

\section{Introduction}\label{sec: intro}

In this paper, we describe the KMS states of Toeplitz extensions of the noncommutative
solenoids constructed by Latr\'emoli\`ere and Packer \cite{LatPack}. We prove that the
extreme boundary of the KMS simplex is homeomorphic to a topological solenoid. In recent
years, following Bost and Connes' work \cite{BostConnes} relating KMS theory to the
Riemann zeta function, there has been a great deal of interest in the KMS structure of
$C^*$-algebras associated to algebraic and combinatorial objects. In particular Laca and
Raeburn's results \cite{LacaRaeburn:AM10} about the Toeplitz algebra of the
$ax+b$-semigroup over $\N$ precipitated a surge of activity around computations of KMS
states for Toeplitz-like extensions. Various authors have studied KMS states on Toeplitz
algebras associated to algebraic objects \cite{BaHLR,LacaRaeburnEtAl:JFA2014,CaHR},
directed graphs \cite{KajiwaraWatatani:KJM2013, aHLRS, ChristensenThomsen:JMAA2016},
higher-rank graphs \cite{Yang, aHLRS4, FarsiGillaspyEtAl:JMAA2016}, $C^*$-correspondences
\cite{LacaNeshveyev, Kakariadis:JFA2015}, and topological graphs \cite{AaHR}. The results
suggest that the KMS structure of such algebras for their natural gauge actions
frequently encodes key features of the generating object.

The noncommutative solenoids $\AA^{\mathscr{S}}_\theta$ are $C^*$-algebras introduced by
Latr\'emoli\`ere and Packer in \cite{LatPack}. They are among the first examples of
twisted $C^*$-algebras of non-compactly-generated abelian groups to be studied in detail,
and have interesting representation-theoretic properties \cite{LatPack2, LatPack3}. In
addition to the definition of noncommutative solenoids as twisted group $C^*$-algebras,
Latr\'emoli\`ere and Packer provide a number of equivalent descriptions. The one we are
interested in realises them as direct limits of noncommutative tori. Specifically, given
a positive integer $N$ and a sequence $\theta_n$ of real numbers such that
$N^2\theta_{n+1} - \theta_n$ is an integer for every $n$, there are homomorphisms
$\AA_{\theta_{n}} \to \AA_{\theta_{n+1}}$ that send the canonical unitary generators of
$\AA_{\theta_{n}}$ to the $N$th powers of the corresponding generators of
$\AA_{\theta_{n+1}}$. The noncommutative solenoid for the sequence $\theta = (\theta_n)$
is the direct limit of the $\AA_{\theta_{n}}$ under these homomorphisms. Latr\'emoli\`ere
and Packer's work focusses on features like simplicity, $K$-theory and classification of
noncommutative solenoids.

Here we use Katsura's theory of topological graph $C^*$-algebras \cite{Katsura:class_I}
to introduce a class of Toeplitz extensions $\TT^{\mathscr{S}}_\theta$ of noncommutative
solenoids, realised as direct limits of Toeplitz extensions $\TT(E_{\theta_n})$ of
noncommutative tori, and then study their KMS states. Our main result says that at
inverse temperatures above zero, the extreme boundary of the KMS simplex of
$\TT^{\mathscr{S}}_\theta$ is homeomorphic to the classical solenoid $\mathscr{S}$, and
there is an action of $\mathscr{S}$ on $\TT^{\mathscr{S}}_\theta$ that induces a free and
transitive action on the extreme KMS states. This is further evidence that KMS structure
for Toeplitz-like algebras recovers key features of the underlying generating objects.
Interestingly, this homeomorphism is subtler than one might expect: though the results of
\cite{AaHR} show that the KMS simplex of each approximating subalgebra $\TT(E_{\theta_n})
\subseteq \TT^{\mathscr{S}}_\theta$ has extreme boundary homeomorphic to the circle,
these homeomorphisms are not compatible with the connecting maps in the inductive system.
In fact, none of the extreme points in the KMS simplex of any $\TT(E_{\theta_n})$ extend
to KMS states of $\TT^{\mathscr{S}}_\theta$. Identifying the simplex of KMS states of a
given $\TT(E_{\theta_n})$ that do extend to KMS states of $\TT^{\mathscr{S}}_\theta$
requires a careful analysis of the interaction between the subinvariance relation,
described in \cite{AaHR}, that characterises KMS states on the $\TT(E_{\theta_n})$ and
the compatibility relation imposed by the connecting maps $\TT(E_{\theta_n})
\hookrightarrow \TT(E_{\theta_{n+1}})$. We think the ideas involved in this analysis may
be applicable to other investigations of KMS states on direct-limit $C^*$-algebras. Our
main result also shows that at inverse temperature~0 there is a unique KMS state (unless
all the $\theta_n$ are zero, a degenerate case that we discuss separately), and that
there are no KMS states at inverse temperatures below zero. Perhaps surprisingly, for
nonzero $\theta$ the structure of the KMS simplex of $\TT^{\mathscr{S}}_\theta$ does not
depend on whether the $\theta_n$ are rational.

We proceed as follows. After a brief preliminaries section, we begin in
Section~\ref{sec:KMSfordirectlimits} by considering KMS states for actions on direct
limits that preserve the approximating subalgebras. We record a general---and presumably
well known---description of the KMS simplex as a projective limit of the KMS simplices of
the approximating subalgebras. The connecting maps in this projective system need not be
surjective, which is the cause of the subtleties that arise in computing the KMS states
of Toeplitz noncommutative solenoids later in the paper. In Section~\ref{sec:rotations},
we consider the topological graph $E_\gamma$ that encodes rotation on the circle $\R/\Z$
by angle $\gamma \in \R$. We describe the Toeplitz algebra $\TT(E_\gamma)$ of this
topological graph as universal for an isometry $S$ and a representation $\pi$ of
$C(\R/\Z)$, and its topological-graph $C^*$-algebra $\OO(E_\gamma)$ as the quotient by
the ideal generated by $1 - SS^*$. In particular, $\OO(E_\gamma)$ is canonically
isomorphic to the noncommutative torus $\AA_\gamma$. In Section~\ref{sec: alt descrip of
NCS} we consider a sequence $\theta = (\theta_n)$ in $\R/\Z$ such that $N^2\theta_{n+1} =
\theta_n$ for all $n$. We use our description of $\TT(E_\gamma)$ from the preceding
section to describe homomorphisms $\psi_n : \TT(E_{\theta_n}) \to \TT(E_{\theta_{n+1}})$
that descend through the quotient maps to the homomorphisms $\tau_n : \OO(E_{\theta_n})
\to \OO(E_{\theta_{n+1}})$ for which the noncommutative solenoid
$\AA^{\mathscr{S}}_\theta$ is isomorphic to $\varinjlim(\OO(E_{\theta_n}), \tau_n)$.

In Section~\ref{sec: the TNCS}, we define the Toeplitz noncommutative solenoid as
$\TT^{\mathscr{S}}_\theta := \varinjlim(\TT(E_{\theta_n}), \psi_n)$, by analogy with the
description of $\AA^{\mathscr{S}}_\theta$ outlined in Section~\ref{sec: alt descrip of
NCS}. We describe a dynamics $\alpha$ on $\TT^{\mathscr{S}}_\theta$ built from the gauge
actions on the approximating subalgebras $\TT(E_{\theta_n})$. Though the gauge actions on
the $\TT(E_{\theta_n})$ are all periodic $\R$-actions, the dynamics $\alpha$ is not. We
are interested in the KMS states for this dynamics. The case
$\theta=\mathbf{0}:=(0,0,0,\dots)$ is a degenerate case, and we outline in
Remark~\ref{rmk:theta=0} how to describe the KMS states in this instance by decomposing
both the algebra $\TT^{\mathscr{S}}_{\mathbf{0}}$ and the dynamics $\alpha$ as tensor
products. Since $\theta_n \not= 0$ implies $\theta_{n+1} \not= 0$, we can thereafter
assume, without loss of generality, that every $\theta_n$ is nonzero. In the remainder of
Section~\ref{sec: the TNCS}, we use our results about direct limits from
Section~\ref{sec:KMSfordirectlimits} to realise the KMS simplex of
$\TT^{\mathscr{S}}_\theta$ for $\alpha$ at an inverse temperature $\beta > 0$ as a
projective limit of spaces $\mcon{r_n}$ of probability measures on $\R/\Z$ that satisfy a
suitable subinvariance condition. This involves an interesting interplay between the
subinvariance condition for KMS states on the $\TT(E_{\theta_n})$ obtained from
\cite{AaHR}, and the compatibility condition coming from the connecting maps $\psi_n$. We
believe that this analysis and our analysis of the space $\mcon{r_n}$ in
Section~\ref{sec:submeasures} may be of independent interest from the point of view of
ergodic theory. The theorems in \cite{AaHR} are silent on the case $\beta = 0$, so we
must argue this case separately, and our results for this case in Section~\ref{sec: the
TNCS} appear less sharp than for $\beta > 0$: they show only that the $\KMS_0$-simplex
embeds in the projective limit of the spaces $\mcon{0}$. But we shall see later that the
subinvariance condition at $\beta = 0$ has a unique solution, so that the projective
limit in this case is a one-point set. So our embedding result for $\beta = 0$ is
sufficient to show that there is a unique $\KMS_0$ state.

In Section~\ref{sec:submeasures} we analyse the space $\mcon{r}$ for $r > 0$. We first
construct a measure $m_r$ satisfying the desired subinvariance relation, and then show
that the measures obtained by composing this $m_r$ with rotations are all of the extreme
points of $\mcon{r}$. This yields an isomorphism of $\mcon{r}$ with the space of Borel
probability measures on $\R/\Z$. A key step in our analysis is the characterisation in
\cite{aHLRS} of the subinvariant measures on the vertex set of a simple-cycle graph. We
then turn in Section~\ref{sec:mainproof} to the proof of our main theorem. The key step
is to establish that the connecting maps $\psi_n : \TT(E_{\theta_n}) \to
\TT(E_{\theta_{n+1}})$ induce surjections $\mcon{r_{n+1}} \twoheadrightarrow \mcon{r_n}$
by showing that the induced maps carry extreme points to extreme points.

\section{Preliminaries}\label{sec:background}

In this section we recall the background that we need on topological graphs and their
$C^*$-algebras, as introduced by Katsura in \cite{Katsura:class_I}. We then recall the
notion of a KMS state for a $C^*$-algebra $A$ and dynamics $\alpha$.

\subsection*{Topological graphs and their \texorpdfstring{$C^*$}{C*}-algebras}
For details of the following, see \cite{Katsura:class_I}. A {\em topological graph}
$E=(E^0,E^1,r,s)$ consists of locally compact Hausdorff spaces $E^0$ and $E^1$, a
continuous map $r:E^1\to E^0$, and a local homeomorphism $s:E^1\to E^0$. In
\cite{Katsura:class_I} Katsura constructs from each topological graph $E$ a Hilbert
$C_0(E^0)$-bimodule $X(E)$ and two $C^*$-algebras: the Toeplitz algebra $\TT(E)$ and the
graph $C^*$-algebra $\OO(E)$. In this article we only encounter topological graphs of the
form $E=(Z,Z,\id,h)$, where $h:Z\to Z$ is a homeomorphism of a compact Hausdorff space
$Z$, so we only discuss the details of $X(E)$, $\TT(E)$ and $\OO(E)$ in this setting.

When $E=(Z,Z,\id,h)$, where $Z$ is compact, the module $X(E)$ is a copy of $C(Z)$ as a
Banach space. The left and right actions are given by
\[
g_1\cdot f\cdot g_2(z)=g_1(z)f(z)g_2(h(z)), \quad\text{for $g_1,g_2\in C(Z)$,
$f\in X(E)$},
\]
and the inner product by $\langle f_1,f_2\rangle
(z)=\overline{f_1(h^{-1}(z))}f_2(h^{-1}(z))$, for $f_1,f_2\in X(E)$. We denote by
$\varphi$ the homomorphism $C(Z)\to\LL(X(E))$ implementing the left action. In this case
$\varphi$ is injective.

A {\em representation} of $X(E)$ in a $C^*$-algebra $B$ is a pair $(\psi,\pi)$,
consisting of a linear map $\psi:X(E)\to B$ and a homomorphism $\pi:C(Z)\to B$ satisfying
\[
\psi(f\cdot h)=\psi(f)\pi(h),\quad \psi^*(f)\psi(g)=\pi(\langle
f,g\rangle)\quad\text{and}\quad \psi(h\cdot f)=\pi(h)\psi(f)
\]
for all $f,g\in X(E)$ and $h\in C(Z)$. The Toeplitz algebra $\TT(E)$ is the Toeplitz
algebra of $X(E)$, in the sense of \cite{FowlerRaeburn}, which is the universal
$C^*$-algebra generated by a representation of $X(E)$. We denote by $(i_{X(E)}^1,
i_{X(E)}^0)$ the representation generating $\TT(E)$.

For $f_1, f_2 \in X(E)$ there is an adjointable operator $\Theta_{f_1, f_2} \in
\LL(X(E))$ given by $\Theta_{f_1, f_2}(g) = f_1 \langle f_2, g\rangle_{C(Z)} = f_1 f^*_2
g$. The algebra of generalised compact operators on $X(E)$ is
\[
\KK(X(E)):=\clsp\{\Theta_{f_1,f_2}:f_1,f_2\in X(E)\}.
\]
Since $\Theta_{1,1} = 1_{\LL(X(E))}$, we have $\KK(X(E))=\LL(X(E))$. For a representation
$(\psi,\pi)$ of $X(E)$ in $B$ there is a homomorphism $(\psi,\pi)^{(1)}:\KK(X(E))\to B$
satisfying $(\psi,\pi)^{(1)}(\Theta_{f_1,f_2})=\psi(f_1)\psi(f_2)^*$ (see
\cite[page~202]{Pimsner}).

The graph algebra $\OO(E)$ is the Cuntz--Pimsner algebra of $X(E)$. So $\OO(E)$ is the
quotient of $\TT(E)$ by the ideal generated by
\[
\{(i_{X(E)}^1,i_{X(E)}^0)^{(1)}(\varphi(h))-i_{X(E)}^0(h): h\in C(Z)\},
\]
and is the universal $C^*$-algebra generated by a covariant representation of
$X(E)$---that is, a representation $(\psi,\pi)$ satisfying
\[
(\psi,\pi)^{(1)}(\varphi(h))=\pi(h)\quad\text{for all $h\in C(Z)$}.
\]
We denote the quotient map $\TT(E)\to\OO(E)$ by $q$, and we define
$(j_{X(E)}^1,j_{X(E)}^0) := (q\circ i_{X(E)}^1,q\circ i_{X(E)}^0)$, the covariant
representation generating $\OO(E)$.

\subsection*{KMS states}
For details of the following, see \cite{BRII}. Given a $C^*$-algebra $A$ and an action
$\alpha : \R \to \Aut(A)$, we say that $a \in A$ is \emph{analytic} for $\alpha$ if the
function $t \mapsto \alpha_t(a)$ is the restriction of an analytic function $z \mapsto
\alpha_z(a)$ from $\C$ into $A$. The set of analytic elements is always norm dense in
$A$. A state $\phi$ of $A$ is a \emph{$\KMS_0$-state} if it is an $\alpha$-invariant
trace on $A$. For $\beta \in\R \setminus \{0\}$, a state $\phi$ of $A$ is a
\emph{$\KMS_\beta$-state}, or a KMS-state at inverse temperature $\beta$, for the system
$(A, \alpha)$ if it satisfies the \emph{KMS condition}
\[
\phi(ab) = \phi(b\alpha_{i\beta}(a))\quad\text{ for all analytic $a,b \in A$.}
\]
It suffices to check this condition for all $a, b$ in any $\alpha$-invariant set of
analytic elements that spans a dense subspace of $A$. The collection of
$\KMS_\beta$-states for a dynamics $\alpha$ on a unital $C^*$-algebra $A$ forms a Choquet
simplex, and we will denote it by $\KMS_\beta(A, \alpha)$.

\section{KMS structure of direct limit \texorpdfstring{$C^*$}{C*}-algebras}\label{sec:KMSfordirectlimits}

The $C^*$-algebras of interest to us in this paper are examples of direct-limit
$C^*$-algebras. In this short section we show that the simplex of KMS states of a
direct-limit $C^*$-algebra, for an action that preserves the approximating subalgebras,
is the projective limit of the simplices of KMS states of the approximating subalgebras.

\begin{proposition}\label{prop:KMSanddirectlimits}
Suppose $\beta\in [0,\infty)$, and that $\{(A_j,\varphi_j,\alpha_j):j\in\N\}$ is a
sequence of unital $C^*$-algebras $A_j$, injective unital homomorphisms $\varphi_j:A_j\to
A_{j+1}$, and strongly continuous actions $\alpha_j:\R\to\Aut A_j$ satisfying
$\alpha_{j+1,t}\circ\varphi_j = \varphi_j\circ\alpha_{j,t}$ for all $j\in\N$ and
$t\in\R$. Denote by $A_\infty$ the direct limit $\varinjlim (A_j,\varphi_j)$, and by
$\varphi_{j,\infty}$ the canonical maps $A_j\to A_\infty$ satisfying
$\varphi_{j+1,\infty}\circ \varphi_j = \varphi_{j,\infty}$ for each $j\in\N$. There is a
strongly continuous action $\alpha : \R\to\Aut A_\infty$ satisfying
$\varphi_{j,\infty}\circ \alpha_{j,t} = \alpha_t\circ\varphi_{j,\infty}$ for each
$j\in\N$ and $t\in\R$. Moreover, there is an affine isomorphism from
$\KMS_{\beta}(A_\infty,\alpha)$ onto
$\varprojlim(\text{KMS}_\beta(A_j,\alpha_j),\phi\mapsto\phi\circ\varphi_{j-1})$ that
sends $\phi$ to $(\phi\circ\varphi_{j,\infty})_{j=0}^\infty$.
\end{proposition}	
\begin{proof}
For each $j\in\N$ and $t\in\R$ we have
\[
(\varphi_{j+1,\infty}\circ \alpha_{j+1,t})\circ \varphi_j= \varphi_{j+1,\infty}\circ\varphi_j\circ\alpha_{j,t}=\varphi_{j,\infty}\circ\alpha_{j,t}.
\]
So the universal property of $A_\infty$ gives a homomorphism $\alpha_t:A_\infty\to
A_\infty$ such that $\alpha_t \circ \varphi_{j,\infty} = \varphi_{j,\infty} \circ
\alpha_{j,t}$ for all $j$.

It is straightforward to check that each $\alpha_{t}$ is an automorphism of $A_\infty$
with inverse $\alpha_{-t}$, and that $\alpha:\R\to\Aut A_\infty$ is an action satisfying
$\varphi_{j,\infty}\circ \alpha_{j,t} = \alpha_{t}\circ\varphi_{j,\infty}$. An
$\varepsilon/3$-argument using that each $\alpha_j$ is strongly continuous and that
$\bigcup_j \varphi_{j,\infty}(A_j)$ is dense in $A_\infty$ shows that $\alpha$ is
strongly continuous.

For $j\in\N$ and $\phi\in\KMS_\beta(A_\infty,\alpha)$ define $h_j(\phi):=\phi\circ
\varphi_{j,\infty}$. Since $\KMS_\beta$ states restrict to $\KMS_\beta$ states on
invariant unital subalgebras, $h_j$ maps $\KMS_\beta(A_\infty,\alpha)$ to
$\KMS_\beta(A_j,\alpha_j)$ for each $j$. We have
\[	
h_{j+1}\circ \varphi_j =(\phi\circ \varphi_{j+1,\infty})\circ\varphi_j=\phi\circ (\varphi_{j+1,\infty}\circ\varphi_j)=\phi\circ \varphi_{j,\infty}=h_j,
\]
and so the universal property of $\varprojlim \KMS_\beta(A_j,\alpha_j)$ gives a map $h$
from $\KMS_\beta(A_\infty,\alpha)$ into $\varprojlim \KMS_\beta(A_j,\alpha_j)$ satisfying
$p_j\circ h = h_j$, where $p_j$ denotes the canonical projection onto
$\KMS_\beta(A_j,\beta_j)$. We claim that $h$ is the desired affine isomorphism.

The map $h$ is obviously affine. To see that $h$ is surjective, fix
$(\phi_j)_{j=0}^\infty\in\varprojlim(\KMS_\beta(A_j,\alpha_j))$, and take $j\le k$, $a\in
A_j$ and $b\in A_k$ with $\varphi_{j,\infty}(a)=\varphi_{k,\infty}(b)$. Then
\[
0=\varphi_{k,\infty}(b)-\varphi_{j,\infty}(a)=\varphi_{k,\infty}(b)-\varphi_{k,\infty}(\varphi_{k-1}\circ\dots\circ\varphi_j(a))=\varphi_{k,\infty}(b-\varphi_{k-1}\circ\dots\circ\varphi_j(a)).
\]
Since each $\varphi_j$ is injective, each $\varphi_{j,\infty}$ is injective, and so
$b=\varphi_{k-1}\circ\dots\circ\varphi_j(a)$. Now
\[
\phi_j(a)=\phi_k(\varphi_{k-1}\circ\dots\circ\varphi_j(a))=\phi_k(b),
\]
and so there is a well-defined linear map
$\phi_\infty:\bigcup_{j=0}^\infty\varphi_{j,\infty}(A_j)\to\C$ satisfying
$\phi_\infty(\varphi_{j,\infty}(a))=\phi_j(a)$ for all $j\in\N$ and $a\in A_j$. Since
each $\varphi_{j,\infty}$ is isometric and each $\phi_j$ is norm-decreasing, each
$\phi_\infty \circ \varphi_{j,\infty}$ is norm-decreasing, so $\phi_\infty$ is
norm-decreasing. It therefore extends to a norm-decreasing $\phi_\infty : A_\infty \to
\C$. Since $\|\phi_\infty\| \ge \|\phi_\infty\circ\varphi_j\| = \|\phi_j\| = 1$, we see
that $\|\phi_\infty\| = 1$. Since $\bigcup_j \varphi_{j,\infty}\big((A_j)_+\big)$ is
dense in $(A_\infty)_+$ and since each $\phi_\infty \circ \varphi_{j,\infty} = \phi_j$ is
positive, $\phi_\infty$ is positive, and therefore a state of $A_\infty$.

To see that $\phi_\infty$ is KMS, observe that if $a \in A_j$ is $\alpha_j$-analytic,
then $\varphi_{j,\infty}(a)$ is $\alpha$-analytic. Indeed, since $z \mapsto
\varphi_{j,\infty}(\alpha_{j,z}(a))$ is an analytic extension of $t \mapsto
\alpha_{t}(\varphi_{j,\infty}(a))$, the analytic extension of $t \mapsto
\alpha_t(\varphi_{j,\infty}(a))$ is given by
\[
    \alpha_{z}(\varphi_{j,\infty}(a)) = \varphi_{j,\infty}(\alpha_{j,z}(a)).
\]
So $\bigcup_j \{\varphi_{j,\infty}(a) : a \in A_j \text{ is analytic}\}$ is an
$\alpha$-invariant dense subspace of analytic elements in $A_\infty$. So it suffices to
show that $\phi_{\infty}\big(\varphi_{j,\infty}(a)\varphi_{k,\infty}(b)\big) =
\phi_{\infty}\big(\varphi_{k,\infty}(b)\alpha_{i\beta}(\varphi_{j,\infty}(a))\big)$
whenever $a \in A_j$ and $b \in A_k$ are analytic. For this, let $l := \max\{j,k\}$ and
observe that $a' := \varphi_{j,l}$ and $b' := \varphi_{k,l}(b)$ are $\alpha_l$-analytic,
and so
\begin{align*}
\phi_\infty(\varphi_{j,\infty}(a)\varphi_{k,\infty}(b))
    &= \phi_\infty(\varphi_{l,\infty}(a'b')
    =\phi_l(a'b')
    = \phi_l(b'\alpha_{l,i\beta}(a'))\\
    &= \phi_\infty\big(\varphi_{l,\infty}(b')\alpha_{i\beta}(\varphi_{l,\infty}(a'))\big)
    = \phi_\infty\big(\varphi_{j,\infty}(b)\alpha_{i\beta}(\varphi_{j,\infty}(a))\big).
\end{align*}
Since $h(\phi_\infty)=(\phi_\infty\circ \varphi_{j,\infty})_{j=0}^\infty =
(\phi_j)_{j=0}^\infty$, we see that $h$ is surjective.

Checking that $h$ is injective is straightforward: if $h(\phi)=h(\psi)$, then
$\phi\circ\varphi_{j,\infty}=\psi\circ\varphi_{j,\infty}$ for all $j\in \N$, which
implies that $\phi$ and $\psi$ agree on the dense subset
$\bigcup_{j=0}^\infty\varphi_{j,\infty}(A_j)$, giving $\phi=\psi$. 	

To see that $h$ is continuous, let $(\phi_{\lambda})_{\lambda\in\Lambda}$ be a net in
$\KMS_\beta(A_\infty,\alpha)$ converging weak* to $\phi\in \KMS_\beta(A_\infty,\alpha)$.
Then $p_j(h(\phi_\lambda))=\phi_\lambda\circ\varphi_{j,\infty}$ converges weak* to
$p_j(h(\phi))=\phi\circ\varphi_{j,\infty}$ for each $j\in\N$. Since the topology on the
inverse limit is the initial topology induced by the projections $p_j$, this says that
$h(\phi_\lambda)$ converges weak* to $h(\phi)$. Hence $h$ is continuous.
\end{proof}

\section{\texorpdfstring{$C^*$}{C*}-algebras from rotations on the circle}\label{sec:rotations}
We are interested in topological graphs built from rotations on the circle. We write
\[
\SI := \R/\Z
\]
for the circle, which we frequently identify with $[0,1)$ under addition modulo 1.

For $\gamma\in\R$, let $R_\gamma$ denote clockwise rotation of the circle $\SI$ by angle
$\gamma$. So $R_\gamma(t) = t - \gamma~(\operatorname{mod}~1)$. Each $R_\gamma$ is a
homeomorphism of $\SI$, and we denote by $E_\gamma:=(\SI,\SI,\id_{\SI},R_\gamma)$ the
corresponding topological graph. We denote the Hilbert bimodule $X(E_\gamma)$ by
$C(\SI)_\gamma$, its inner product by $\langle\cdot,\cdot\rangle_\gamma$, and the
homomorphism implementing the left action by $\phi_\gamma:C(\SI)\to \LL(C(\SI)_\gamma)$.

We can give alternative characterisations of the $C^*$-algebras $\TT(E_\gamma)$ and
$\OO(E_\gamma)$. This is certainly not new: the description of $\OO(E_\gamma)$ goes back
to Pimsner \cite[page~193, Example~3]{Pimsner}. But we could not locate the exact
formulation that we want for the description of $\TT(E_\gamma)$ in the literature.

\begin{definition}\label{def: TC&CPairs}
A {\em Toeplitz pair for $E_\gamma$ in a $C^*$-algebra $B$} is a pair $(\pi,S)$
consisting of a homomorphism $\pi$ of $C(\SI)$ into $B$, and an isometry $S\in B$
satisfying
\[
S\pi(f)=\pi(f\circ R_\gamma)S\quad\text{for all $f\in C(\SI)$}.
\]
A {\em covariant pair for $E_\gamma$} is a Toeplitz pair $(\pi,W)$ in which $W$ is a
unitary.
\end{definition}

\begin{proposition}\label{prop:TC&CPairs C*s}
Let $\gamma\in\R$ and $E_\gamma=(\SI,\SI,\id_{\SI},R_\gamma)$.
\begin{enumerate}
\item[(1)] The pair $(i_\gamma,s_\gamma):=(i_{X(E_\gamma)}^0,i_{X(E_\gamma)}^1(1))$
    is a Toeplitz pair for $E_\gamma$ that generates $\TT(E_\gamma)$. Moreover,
    $\TT(E_\gamma)$ is the universal $C^*$-algebra generated by a Toeplitz pair for
    $E_\gamma$: if $(\pi,S)$ is a Toeplitz pair in a $C^*$-algebra $B$, then there is
    a homomorphism $\pi\times S:\TT(E_\gamma)\to B$ satisfying $(\pi\times S)\circ
    i_\gamma=\pi$ and $(\pi\times S)(s_\gamma)=S$.

\item[(2)] The pair $(j_\gamma,w_\gamma):=(j_{X(E_\gamma)}^0,j_{X(E_\gamma)}^1(1))$
    is a covariant pair for $E_\gamma$ that generates $\OO(E_\gamma)$. Moreover,
    $\OO(E_\gamma)$ is the universal $C^*$-algebra generated by a covariant pair for
    $E_\gamma$: if $(\pi,W)$ is a covariant pair in a $C^*$-algebra $B$, then there
    is a homomorphism $\pi\times W:\OO(E_\gamma)\to B$ satisfying $(\pi\times W)\circ
    j_\gamma=\pi$ and $(\pi\times W)(w_\gamma)=W$.
\end{enumerate}
\end{proposition}
\begin{proof}
We have $s_\gamma^*s_\gamma=i_{X(E_\gamma)}^0(\langle
1,1\rangle_\gamma)=i_{X(E_\gamma)}^0(1)=1$, and so $s_\gamma$ is an isometry. For each
$f\in C(\SI)$ we have
\begin{align*}
i_\gamma(f\circ R_\gamma)s_\gamma
    &= i_{X(E_\gamma)}^0(f\circ R_\gamma)i_{X(E_\gamma)}^1(1)
     = i_{X(E_\gamma)}^1((f\circ R_\gamma)\cdot1)\\
    &= i_{X(E_\gamma)}^1(1\cdot f)
     =i_{X(E_\gamma)}^1(1)i_{X(E_\gamma)}^0(f)=s_\gamma i_\gamma(f),
\end{align*}
and so $(i_\gamma,s_\gamma)$ is a Toeplitz pair. For $f\in C(\SI)_\gamma$ we have
$i_{X(E_\gamma)}^1(f)=i_\gamma(f)s_\gamma$, so the pair $(i_\gamma,i_\eta^1(1))$
generates the ranges of both $i_{X(E_\gamma)}^0$ and $i_{X(E_\gamma)}^1$, and hence all
of $\TT(E_\gamma)$.

Now suppose $B$ is a unital $C^*$-algebra and $\pi:C(\SI)\to B$ and $S\in B$ form a
Toeplitz pair $(\pi,S)$ for $E_\gamma$ in $B$. Define $\psi:C(\SI)_\gamma\to B$ by
$\psi(f)=\pi(f)S$. We claim that $(\psi,\pi)$ is a representation of $C(\SI)_\gamma$ in
$B$. For each $f\in C(\SI)_\gamma$ and $g\in C(\SI)$ we have
\[
\pi(g)\psi(f)=\pi(g)\pi(f)S=\pi(gf)S=\psi(gf)=\psi(g\cdot f)
\]
and
\[
\psi(f)\pi(g)=\pi(f)S\pi(g)=\pi(f(g\circ R_\gamma))S=\psi(f(g\circ
R_\gamma))=\psi(f\cdot g).
\]
To check that the inner product is preserved, we let $f,h\in C(\SI)_\gamma$ and calculate
\begin{align*}
\psi(f)^*\psi(h)&=S^*\pi(f^*)\pi(h)S
     = S^*\pi(f^*\circ R_\gamma^{-1}\circ R_\gamma)\pi(h\circ R_\gamma^{-1}\circ R_\gamma)S\\
    &= \pi(f^*\circ R_\gamma^{-1})S^*S\pi(h\circ R_\gamma^{-1})
     = \pi((f^*h)\circ R_\gamma^{-1}).
\end{align*}
We have $\langle
f,h\rangle_\gamma(z)=\overline{f(R_\gamma^{-1}(z))}g(R_\gamma^{-1}(z))=(f^*g)\circ
 R_\gamma^{-1}(z)$. So $\langle f,h\rangle_\gamma=(f^*h)\circ R_\gamma^{-1}$,
and hence $\psi(f)^*\psi(h)=\pi(\langle f,h\rangle_\gamma)$. This proves the claim.

The universal property of $\TT(E_\gamma)$ yields a homomorphism
$\psi\times\pi:\TT(E_\gamma)\to B$ satisfying $(\psi\times\pi)\circ
i_{X(E_\gamma)}^1=\psi$ and $(\psi\times\pi)\circ i_{X(E_\gamma)}^0=\pi$. Let $\pi\times
S:=\psi\times\pi$. Then
\[
    (\pi\times S)\circ i_{X(E_\gamma)}^0 = (\psi\times\pi)\circ i_{X(E_\gamma)}^0 = \pi,
\]
and
\[
    (\pi\times S)(s_\gamma) = (\psi\times\pi)(i_{X(E_\gamma)}^1(1)) = \psi(1) = \pi(1)S = S.
\]
Hence $\TT(E_\gamma)$ is the universal $C^*$-algebra generated by a Toeplitz pair for
$E_\gamma$.

To prove (2) it suffices to show that the ideal $I$ generated by
\[
\{(i_{X(E_\gamma)}^1,i_{X(E_\gamma)}^0)^{(1)}(\varphi_\gamma(f))-i_{X(E_\gamma)}^0(f):f\in
 C(\SI)\}
\]
is the ideal generated by the element $s_\gamma s_\gamma^*-1$. We have
\[
s_\gamma s_\gamma^*-1=
(i_{X(E_\gamma)}^1,i_{X(E_\gamma)}^0)^{(1)}(\Theta_{1,1})-i_{X(E_\gamma)}^0
=
(i_{X(E_\gamma)}^1,i_{X(E_\gamma)}^0)^{(1)}(\phi_\eta(1))-i_{X(E_\gamma)}^0(1)\in
 I,
\]
and hence the ideal generated by $s_\gamma s_\gamma^*-1$ is contained in $I$. For the
reverse containment we first note that $\varphi_\gamma(f)=\Theta_{f,1}$ for all $f\in
C(\SI)$. Then
\begin{align*}
(i_{X(E_\gamma)}^1,i_{X(E_\gamma)}^0)^{(1)}(\varphi_\eta(f))-i_\eta^0(f)
    &= (i_{X(E_\gamma)}^1,i_{X(E_\gamma)}^0)^{(1)}(\Theta_{f,1})-i_\eta^0(f)\\
    &=i_{X(E_\gamma)}^1(f)i_{X(E_\gamma)}^1(1)^*-i_\eta^0(f)\\
    &= i_{X(E_\gamma)}^0(f)i_{X(E_\gamma)}^1(1)i_{X(E_\gamma)}^1(1)^*-i_{X(E_\gamma)}^0(f)\\
    &=i_{X(E_\gamma)}^0(f)\big(s_\gamma s_\gamma^*-1\big),
\end{align*}
and the result follows.
\end{proof}

\begin{remarks}\label{rems:Uniqueness thm for TT}
\begin{enumerate}
\item We saw in the proof of Proposition~\ref{prop:TC&CPairs C*s} that a Toeplitz
    pair $(\pi,S)$ for $E_\gamma$ gives a representation $(\psi,\pi)$ of
    $X(E_\gamma)$ such that $\psi(f)=\pi(f)S$. We denote the homomorphism
    $(\psi,\pi)^{(1)}$ of $\KK(X(E_\gamma))$ by $(\pi,S)^{(1)}$; so
    $(\pi,S)^{(1)}(\Theta_{f,g})=\pi(f)SS^*\pi(g)^*$.
\item In \cite[Theorem~6.2]{Kat} Katsura proved a gauge-invariant uniqueness theorem
    for the Toeplitz algebra of a Hilbert bimodule. Suppose $A$ is a $C^*$-algebra,
    $X$ is a Hilbert $A$-bimodule, and $(\psi,\pi)$ is a representation of $X$ in a
    $C^*$-algebra $B$. The gauge-invariant uniqueness theorem says that
    $\psi\times\pi:\TT(X)\to B$ is injective if $B$ carries a gauge action,
    $\psi\times \pi$ intertwines the gauge actions on $\TT(X)$ and $B$, and the ideal
\[
    \{a\in A : \pi(a)\in (\psi,\pi)^{(1)}(\KK(X))\}
\]
of $A$ is trivial. If $(\pi,S)$ is a Toeplitz pair for $E_\gamma$, then this ideal is
$\{f\in C(\SI): \pi(f)\in (\pi,S)^{(1)}(\KK(X(E_\gamma)))\}$, which we can write as
\[
    \{f\in C(\SI):\pi(f)\in\clsp\{\pi(g)SS^*\pi(h) : g,h \in C(\SI)\}\}.
\]
We denote this ideal by $I_{(\pi,S)}$.
\item\label{it:kernel trivial} Proposition~4.10 of \cite{Kat} says that
    $I_{(i_\gamma,s_\gamma)}=0$.
\end{enumerate}
\end{remarks}

We can give spanning families for $\TT(E_\gamma)$ and $\OO(E_\gamma)$ using Toeplitz and
covariant pairs.

\begin{proposition}\label{prop: spanning families}
Let $\gamma\in\R$ and $E_\gamma=(\SI,\SI,\id_\SI,R_\gamma)$. Then
\[
    \TT(E_\gamma)=\clsp\{s_\gamma^mi_\gamma(f) {s_\gamma^*}^n:m,n\in\N,\,f\in C(\SI)\},
\]
and
\[
    \OO(E_\gamma)=\clsp\{w_\gamma^mj_\gamma(f) {w_\gamma^*}^n:m,n\in\N,\,f\in C(\SI)\}.
\]
\end{proposition}
\begin{proof}
The set $\lsp\{s_\gamma^mi_\gamma(f) {s_\gamma^*}^n:m,n\in\N,\,f\in C(\SI)\}$ contains
the generators of $\TT(E_\gamma)$, so it suffices to show that it is a $*$-subalgebra. It
is obviously closed under involution; that it is closed under multiplication follows from
the calculation
\begin{align*}
s_\gamma^mi_\gamma(f){s_\gamma^*}^ns_\gamma^pi_\gamma(g){s_\gamma^*}^q
    &= \begin{cases}
        s_\gamma^mi_\gamma(f){s_\gamma^*}^{n-p}i_\gamma(g){s_\gamma^*}^q & \text{if $n\ge p$}\\
        s_\gamma^mi_\gamma(f)s_\gamma^{p-n}i_\gamma(g){s_\gamma^*}^q & \text{if $n< p$}
    \end{cases}\\
    &= \begin{cases}
        s_\gamma^mi_\gamma(f(g\circ R_\gamma^{-(n-p)})){s_\gamma^*}^{n-p+q} & \text{if $n\ge p$}\\
        s_\gamma^{m+p-n}i_\gamma((f\circ R_\gamma^{-(p-n)})g){s_\gamma^*}^q & \text{if $n< p$.}
    \end{cases}
\end{align*}
Since each $s_\gamma^mi_\gamma(f) {s_\gamma^*}^n$ is mapped to $w_\gamma^mj_\gamma(f)
{w_\gamma^*}^n$ under the quotient map $\TT(E_\gamma)\to \OO(E_\gamma)$, we have
$\OO(E_\gamma)=\clsp\{w_\gamma^mj_\gamma(f) {w_\gamma^*}^n:m,n\in\N,\,f\in C(\SI)\}$.
\end{proof}

\section{An alternative description of the noncommutative solenoid}\label{sec:
alt descrip of NCS}

Throughout the rest of this paper we fix a natural number $N\ge 2$. In \cite{LatPack},
given a sequence $\theta = (\theta_n)^\infty_{n=1}$ in $\SI = \R/\Z$ such that
$N^2\theta_{n+1} = \theta_n$ for all $n$, Latr\'{e}moli\`{e}re and Packer define the
noncommutative solenoid $\AA_\theta^\mathscr{S}$ as a twisted group $C^*$-algebra
involving the $N$-adic rationals. In \cite[Theorem~3.7]{LatPack} they give an equivalent
characterisation of $\AA_\theta^\mathscr{S}$. We will take this characterisation as our
definition. We recall it now. Let
\[
\Xi_N:=\{(\theta_n)_{n=0}^\infty: \theta_n\in\SI\text{ and }N^2\theta_{n+1} = \theta_n\text{ for each $n$}\}.
\]
Recall that for $\gamma\in\SI$ the rotation algebra $\AA_\gamma$ is the universal
$C^*$-algebra generated by unitaries $U_\gamma$ and $V_\gamma$ satisfying $U_\gamma
V_\gamma=e^{2\pi i\gamma}V_\gamma U_\gamma$.

\begin{definition}\label{def: NCS}
Let $\theta=(\theta_n)_{n=0}^\infty\in\Xi_N$, and for each $n\in\N$ let
$\varphi_n:\AA_{\theta_n}\to\AA_{\theta_{n+1}}$ be the homomorphism satisfying
\[
\varphi_n(U_{\theta_n})=U_{\theta_{n+1}}^N\quad\text{and}\quad
\varphi_n(V_{\theta_n})=V_{\theta_{n+1}}^N.
\]
The {\em noncommutative solenoid} $\AA_\theta^\mathscr{S}$ is the direct limit
$\varinjlim (\AA_{\theta_n},\varphi_n)$.
\end{definition}

\begin{remark}\label{rem:changefromLP}
We have taken a slightly different point of view to \cite{LatPack} in describing
$\AA_\theta^\mathscr{S}$. In \cite{LatPack}, Latr\'emoli\`ere and Packer consider
collections of $(\theta_n)$ such that $N\theta_{n+1}-\theta_n\in\Z$, and take the direct
limit $\varinjlim \AA_{\theta_{2n}}$, with intertwining maps going from
$\AA_{\theta_{2n}}$ to $\AA_{\theta_{2n+2}}$.
\end{remark}

	We now give an alternative characterisation of the noncommutative solenoid using
topological graphs built from rotations of the circle as discussed in
Section~\ref{sec:rotations}.

\begin{notation}\label{notation:ioatandthecoveringmap}
We denote by $\iota:\SI\to \T$ the homeomorphism $t\mapsto e^{2\pi it}$,
and by $p_N:\SI\to\SI$ the function $t\mapsto Nt$.
\end{notation}

\begin{proposition}\label{prop: alt descrip of NCS}
Let $N\ge 2$, and $\theta=(\theta_n)_{n=0}^\infty\in\Xi_N$. For each $n\in\N$ there is an
injective homomorphism $\tau_n:\OO(E_{\theta_n})\to \OO(E_{\theta_{n+1}})$ satisfying
\[
\tau_n(j_{\theta_n}(f))=j_{\theta_{n+1}}(f\circ p_N)
\quad\text{and}\quad\tau_n(w_{\theta_n})=w_{\theta_{n+1}}^N,
\]
for all $f\in C(\SI)$. Moreover $\varinjlim (\OO(E_{\theta_n}),\tau_n) \cong
\AA_\theta^\mathscr{S}$.
\end{proposition}

We will prove the existence of the injective homomorphisms $\tau_n$ using the following
result.

\begin{lemma}\label{lem: map from gamma to eta for Toeplitz}
Let $N\in\N$ with $N\ge 2$, and take $\gamma,\eta\in\SI$ with $N^2\eta-\gamma\in\Z$. Then
there is an injective homomorphism
$\psi:\TT(E_\gamma)\to \TT(E_\eta)$ satisfying
\[
    \psi(i_{\gamma}(f)) = i_{\eta}(f\circ p_N)
        \quad\text{and}\quad
    \psi(s_{\gamma})=s_{\eta}^N,
\]
for all $f\in C(\SI)$. The map $\psi$ descends to an injective homomorphism
$\tau:\OO(E_\gamma)\to \OO(E_\eta)$ satisfying
$\tau(j_{\gamma}(f))=j_{\eta}(f\circ p_N)$ and $\tau(w_{\gamma})=w_{\eta}^N$ for all
$f\in C(\SI)$.
\end{lemma}
\begin{proof}
Consider $\pi:C(\SI)\to \TT(E_\eta)$ given by $\pi(f)=i_\eta(f\circ p_N)$ and let
$S:=s_\eta^N$. Since
\[
R_\gamma\circ p_N
    = R_{N^2\eta}\circ p_N
    = R_\eta^{N^2}\circ p_N=p_N\circ R_\eta^N,
\]
we have
\[
\pi(f\circ R_\gamma)S=i_\eta(f\circ R_\gamma\circ p_N)s_\eta^N
    = i_\eta(f\circ p_N\circ R_\eta^N)s_\eta^N
    = s_\eta^Ni_\eta(f\circ p_N)
    = S\pi(f).
\]
So $(\pi,S)$ is a Toeplitz pair for $E_\gamma$. The universal property of $\TT(E_\gamma)$
now gives a homomorphism $\psi:\TT(E_\gamma)\to \TT(E_\eta)$ satisfying
$\psi_n(i_{\gamma}(f))=i_{\eta}(f\circ p_N)$ for all $f\in C(\SI)$, and
$\psi(s_{\gamma})=s_{\eta}^N$.

To see that $\psi$ is injective, we aim to apply the gauge-invariant uniqueness theorem
discussed in Remarks~\ref{rems:Uniqueness thm for TT}. We claim that
$I_{(\pi,S)}\not=0\implies I_{(i_\eta,s_\eta)}\not=0$. To see this, suppose that $0\not=
f\in I_{(\pi,S)}$. Fix $\epsilon>0$, and choose $g_i,h_i\in C(\SI)$ with
\[
\Big\|\pi(f) - \sum_{i=1}^k\pi(g_i)SS^*\pi(h_i)\Big\|<\epsilon.
\]
So
\[
\Big\|i_\eta(f\circ p_N) - \sum_{i=1}^ki_\eta(g_i\circ p_N)s_\eta^N{s_\eta^*}^Ni_\eta(h_i\circ p_N)\Big\|<\epsilon.
\]
For every function $g\in C(\SI)$ we have
\[
i_\eta(g\circ R_{-\eta}^{N-1})={s_\eta^*}^{N-1}s_\eta^{N-1}i_\eta(g\circ R_{-\eta}^{N-1})={s_\eta^*}^{N-1}i_\eta(g)s_\eta^{N-1}.
\]
Hence
\begin{align*}
\Big\|i_\eta(f\circ p_N\circ R_{-\eta}^{N-1})&{} - \sum_{i=1}^ki_\eta(g_i\circ p_N\circ R_{-\eta}^{N-1})s_\eta s_\eta^*i_\eta(h_i\circ p_N\circ R_{-\eta}^{N-1})\Big\|\\
    & = \Big\|{s_\eta^*}^{N-1}i_\eta(f\circ p_N){s_\eta}^{N-1} - \sum_{i=1}^k{s_\eta^*}^{N-1}i_\eta(g_i\circ p_N)s_\eta^N {s_\eta^*}^Ni_\eta(h_i\circ p_N){s_\eta}^{N-1}\Big\|\\
    & \le \Big\|i_\eta(f\circ p_N) - \sum_{i=1}^ki_\eta(g_i\circ p_N)s_\eta^N {s_\eta^*}^Ni_\eta(h_i\circ p_N)\Big\|
    < \epsilon.
\end{align*}
It follows that $i_\eta(f\circ p_N\circ R_{-\eta}^{N-1})\in
(i_\eta,s_\eta)^{(1)}(\KK(X(E_\eta)))$, and hence that $f\circ p_N\circ
R_{-\eta}^{N-1}\in I_{(i_\eta,s_\eta)}$. This proves the claim.

By Remarks~\ref{rems:Uniqueness thm for TT}(\ref{it:kernel trivial}),
$I_{(i_\eta,s_\eta)}=0$, so the claim gives $I_{(\pi,S)}=0$. We have $\psi(\TT(E_\gamma))
\subseteq \clsp\{s_\eta^{aN} i_\eta(f) s_\eta^{*bN} : f \in C(\SI), a,b \in \N\}$. Hence
the gauge action $\rho^\eta$ of $\T$ on $\TT(E_\gamma)$ satisfies $\rho^\eta_z \circ \psi
= \rho^\eta_{z + e^{2\pi i/N}} \circ \psi$ for all $z \in \T$. So there is an action
$\tilde\rho^\eta$ of $\T$ on $\psi(\TT(E_\gamma))$ such that $\tilde\rho^\eta_{e^{2\pi i
t}} \circ \psi = \rho{e^{e\pi i t/N}}$ for all $t \in \R$. In particular,
\begin{align*}
\tilde\rho^\eta_{e^{2\pi i t}} \circ \psi(s_\gamma^{a} i_\eta(f) s_\gamma^{*b})
    &= \rho^\eta_{e^{2\pi i t/N}}(s_\eta^{aN} i_\eta(f) s_\eta^{*bN})\\
    &= e^{2\pi i t(a-b)} s_\eta^{aN} i_\eta(f) s_\eta^{*bN}
    = \psi \circ \rho^\gamma_{e^{2\pi i t}}(s_\eta^{aN} i_\eta(f) s_\eta^{*bN}).
\end{align*}
So continuity and linearity gives $\tilde\rho^\eta_{e^{2\pi i t}} = \psi \circ
\rho^\gamma_{e^{2\pi i t}}$. Hence the gauge-invariant uniqueness theorem
\cite[Theorem~6.2]{Kat} shows that $\psi$ is injective.

To see that $\psi$ descends to the desired injective homomorphism $\tau:\OO(E_\gamma)\to
\OO(E_\eta)$, it suffices to show that the image under $\psi$ of the kernel of the
quotient map $\TT(E_\gamma)\to \OO(E_\gamma)$ is contained in the kernel of
$\TT(E_\eta)\to \OO(E_\eta)$. For this, it suffices to show that $\psi(1-s_\gamma
s_\gamma^*)$ is in the ideal generated by $1-s_\eta s_\eta^*$, which it is because
\[
\psi(1-s_\gamma s_\gamma^*)=1-s_\eta^N {s_\eta^*}^N = \sum_{i=1}^N s_\eta^{N-i}(1-s_\eta s_\eta){s_\eta^*}^{N-i}.\qedhere
\]
\end{proof}

\begin{proof}[Proof of Proposition~\ref{prop: alt descrip of NCS}]
For each $n\in\N$, Lemma~\ref{lem: map from gamma to eta for Toeplitz} applied to
$\gamma=\theta_n$ and $\eta=\theta_{n+1}$ gives the desired injective homomorphism
$\tau_n$.

Proposition~\ref{prop:TC&CPairs C*s} says that each $\OO(E_\eta)$ is the crossed product
$C(\SI)\rtimes \Z$ for the automorphism $f\mapsto f\circ R_\eta$ of $C(\SI)$, which is
the rotation algebra $\AA_\eta$ (see \cite[Example VIII.1.1]{Davidson} for details). So
for each $n\in\N$ there is an isomorphism from $\OO(E_{\theta_n})$ to $\AA_{\theta_n}$
carrying $j_{\theta_n}(\iota)$ to $U_{\theta_n}$ and $w_{\theta_n}$ to $V_{\theta_n}$.
Since each $\tau_n$ satisfies
\[
\tau_n(j_{\theta_n}(\iota))=j_{\theta_{n+1}}(\iota\circ p_N)
=j_{\theta_{n+1}}(\iota)^N\quad\text{and}\quad
\tau_n(w_{\theta_n})=w_{\theta_{n+1}}^N,
\]
the diagrams
\[
\begin{tikzpicture}[scale=1.5]
\node (A) at (0,0) {$\OO(E_{\theta_n})$};
\node (B) at (2,0) {$\OO(E_{\theta_{n+1}})$};
\node (C) at (0,-1) {$\AA_{\theta_n}$};
\node (D) at (2,-1) {$\AA_{\theta_{n+1}}$};
\path[->,font=\scriptsize]
(A) edge node[above]{$\tau_n$} (B)
(A) edge node[left]{$\cong$} (C)
(B) edge node[right]{$\cong$} (D)
(C) edge node[above]{$\varphi_n$} (D);
\end{tikzpicture}
\]
commute. Hence $\varinjlim (\OO(E_{\theta_n}),\tau_n) \cong \AA_\theta^\mathscr{S}$.
\end{proof}

\begin{remark}\label{rem: maps not from graph morphisms}
In \cite[Section~2]{Katsura:IJM06}, Katsura studies factor maps between topological-graph
$C^*$-algebras, and the $C^*$-homomorphisms that they induce. He shows that the
projective limit of a sequence $(E_n)$ of topological graphs under factor maps is itself
a topological graph. He then proves that the $C^*$-algebra $\OO(\varprojlim E_n)$ of this
topological graph is isomorphic to the direct limit $\varinjlim \OO(E_n)$ of the
$C^*$-algebras of the $E_n$ under the homomorphisms induced by the factor maps. So it is
natural to ask whether the maps $\tau_n : \OO(E_{\theta_n}) \to \OO(E_{\theta_{n+1}})$
correspond to factor maps. This is not the case: as observed on page~88 of
\cite{HawkinsPhD}, there is no factor map from $E_{\theta_{n+1}} \to E_{\theta_n}$ that
induces the homomorphism of $C^*$-algebras described in Lemma~\ref{lem: map from gamma to
eta for Toeplitz}.
\end{remark}

\section{The Toeplitz noncommutative solenoid and its KMS structure}\label{sec: the
TNCS}

In this section we introduce our Toeplitz noncommutative solenoids
$\TT^\mathscr{S}_\theta$. We introduce a natural dynamics on $\TT^\mathscr{S}_\theta$ and
apply Proposition~\ref{prop:KMSanddirectlimits} to begin to study its KMS structure.

Given $\theta=(\theta_n)_{n=0}^\infty\in\Xi_N$, Lemma~\ref{lem: map from gamma to eta for
Toeplitz} gives a sequence of injective homomorphisms $\psi_n:\TT(E_{\theta_n})\to
\TT(E_{\theta_{n+1}})$ satisfying
\[
\psi_n(i_{\theta_n}(f))=i_{\theta_{n+1}}(f\circ p_N)
\quad\text{and}\quad\psi_n(s_{\theta_n})=s_{\theta_{n+1}}^N,
\]
for all $f\in C(\SI)$.

\begin{definition}\label{def: the TNCS}
We define $\TT_\theta^\mathscr{S}:=\varinjlim (\TT(E_{\theta_n}),\psi_n)$ and call it the
{\em Toeplitz noncommutative solenoid}. We write
$\psi_{n,\infty}:\TT(E_{\theta_n})\to\TT_\theta^\mathscr{S}$ for the canonical
inclusions, so that $\psi_{n,\infty}=\psi_{n+1,\infty}\circ \psi_n$ for all $n$.
\end{definition}

The following lemma indicates why it is sensible to regard $\TT^{\mathscr{S}}_\theta$ as
a natural Toeplitz extension of the noncommutative solenoid.

\begin{lemma}\label{lem:kernel}
In the notation established in Proposition~\ref{prop: alt descrip of NCS}, there is a
surjective homomorphism $q : \TT_\theta^\mathscr{S} \to \AA^{\mathscr{S}}_\theta$ such
that $q(\psi_{n,\infty}(i_{\theta_n}(f))) = \tau_{n,\infty}(j_{\theta_n}(f))$ and
$q(\psi_{n,\infty}(s_{\theta_n})) = \tau_{n,\infty}(w_{\theta_n})$ for all $n \in \N$ and
all $f \in C(\SI)$. Moreover, $\ker(q)$ is generated as an ideal by
$\psi_{1,\infty}(i_{\theta_1}(1) - s_{\theta_1} s^*_{\theta_1})$.
\end{lemma}
\begin{proof}
For the first statement observe that the canonical homomorphisms $q_n : \TT(E_{\theta_n})
\to \OO(E_{\theta_n})$ intertwine the $\psi_n$ with the $\tau_n$. For the second
statement, let $I$ be the ideal of $\TT^{\mathscr{S}}_\theta$ generated by
$\psi_{1,\infty}(i_{\theta_1}(1) - s_{\theta_1} s^*_{\theta_1})$. Since $\ker(q)$ clearly
contains $\psi_{1,\infty}(i_{\theta_1}(1) - s_{\theta_1} s^*_{\theta_1})$, we have $I
\subseteq \ker(q)$. For the reverse inclusion, note that for $n \ge 1$,
\begin{align*}
\psi_{1,n}(i_{\theta_1}(1) - s_{\theta_1} s^*_{\theta_1})
    &= i_{\theta_n}(1 \circ \iota_{N^n}) - s_{\theta_n}^{nN} s_{\theta_n}^{*nN}\\
    &= i_{\theta_n}(1) - s_{\theta_n} (s_{\theta_n}^{nN-1} s_{\theta_n}^{*(nN-1)}) s_{\theta_n}^*
    \ge i_{\theta_n}(1) - s_{\theta_n} s_{\theta_n}^*,
\end{align*}
so each $\psi_{n,\infty}(i_{\theta_n}(1) - s_{\theta_n} s_{\theta_n}^*) \le
\psi_{n,\infty}(\psi_{1,n}(i_{\theta_1}(1) - s_{\theta_1} s^*_{\theta_1})) =
\psi_{1,\infty}(i_{\theta_1}(1) - s_{\theta_1} s^*_{\theta_1})$, which belongs to $I$.
Thus $\psi_{n,\infty}(i_{\theta_n}(1) - s_{\theta_n} s_{\theta_n}^*) \in I$. Since
$\ker(q) = \overline{\bigcup_n \ker(q) \cap \psi_{n,\infty}(\TT(E_{\theta_n}))} =
\overline{\bigcup_n \psi_{n,\infty}(\ker(q_n))}$, it therefore suffices to show that each
$\ker(q_n)$ is generated by $i_{\theta_n}(1) - s_{\theta_n} s_{\theta_n}^*$, which
follows from Proposition~\ref{prop:TC&CPairs C*s}.
\end{proof}

\begin{proposition}\label{prop:dynamicsonTT}
There is a strongly continuous action $\alpha:\R\to \Aut \TT_\theta^\mathscr{S}$
satisfying
\begin{equation}\label{eq:actiononspanningfamily}
\alpha_t(\psi_{j,\infty}(s_{\theta_{j}}^mi_{\theta_{j}}(f)s_{\theta_{j}}^{*n}))
    = e^{it(m-n)/N^{j}} \psi_{j,\infty}(s_{\theta_{j}}^mi_{\theta_{j}}(f)s_{\theta_{j}}^{*n}),
\end{equation}
for each $j,m,n\in \N$ and $f\in C(\SI)$. This $\alpha$ descends to a strongly continuous
action, also written $\alpha$, on the noncommutative solenoid $\AA^{\mathscr{S}}_\theta$.
\end{proposition}
\begin{proof}
For each $j\in\N$ we denote by $\rho$ the gauge action on $\TT(E_{\theta_j})$, and by
$\rho_j$ the strongly continuous action $t\mapsto \rho_{e^{it/N^j}}$ of $\R$ on
$\TT(E_{\theta_j})$; so $\rho_{j,t}\circ i_{\theta_j}=i_{\theta_j}$ and
$\rho_{j,t}(s_{\theta_j})=e^{it/N^j}s_{\theta_j}$ for each $t\in\R$. For each $j\in\N$
and $t\in\R$ we have
\begin{align*}
\rho_{j+1,t}\circ \psi_j(s_{\theta_j}^m i_{\theta_j}(f)s_{\theta_j}^{*n})
    &= e^{it(Nm-Nn)/N^{j+1}}s_{\theta_{j+1}}^{Nm} i_{\theta_{j+1}}(f)s_{\theta_{j+1}}^{*Nn}\\
    &= e^{it(m-n)/N^j}s_{\theta_{j+1}}^{Nm} i_{\theta_{j+1}}(f)s_{\theta_{j+1}}^{*Nn}
	= \psi_j\circ\rho_{j,t}(s_{\theta_j}^m i_{\theta_j}(f)s_{\theta_j}^{*n}).
\end{align*}
Hence $\rho_{j+1,t}\circ \psi_j=\psi_j\circ\rho_{j,t}$, and
Proposition~\ref{prop:KMSanddirectlimits} applied to each
$(A_j,\alpha_j)=(\TT(E_{\theta_j}),\rho_j)$ gives the desired action $\alpha : \R\to \Aut
\TT_\theta^\mathscr{S}$.

For the final statement, observe that the $\alpha_{t}$ all fix
$\psi_{1,\infty}(i_{\theta_1}(1) - s_{\theta_1} s^*_{\theta_1})$, and so leave the ideal
that it generates invariant; so they descend to $\AA^{\mathscr{S}}_\theta$ by
Lemma~\ref{lem:kernel}.
\end{proof}

\begin{remark}\label{rem:realactionofR}
The actions on graph $C^*$-algebras and their analogues studied in, for example,
\cite{ChristensenThomsen:JMAA2016, EnomotoFujiiEtAl:MJ84, AaHR, aHLRS} are lifts of
circle actions, and so are periodic in the sense that $\alpha_t = \alpha_{t + 2\pi}$ for
all $t$. By contrast, while the action $\alpha$ of the preceding proposition restricts to
a periodic action on each approximating subalgebra $\psi_{j,\infty}(\TT(E_{\theta_j}))$,
it is itself not periodic: $\alpha_t = \alpha_s \implies t = s$.
\end{remark}

We now wish to study the KMS structure of the Toeplitz noncommutative solenoid
$\TT_\theta^\mathscr{S}$ under the dynamics $\alpha$ of
Proposition~\ref{prop:dynamicsonTT}.

\begin{remark}\label{rmk:theta=0}
The case $\theta = \mathbf{0} = (0, 0, \dots)$ is relatively easy to analyse. Let
$\mathscr{S} = \varprojlim(\SI, p_N)$ denote the classical solenoid, and $\TT$ the Toeplitz algebra. Write $s$ for the isometry generating $\TT$, and
$\kappa : \TT \to \TT$ for the homomorphism given by $\kappa(s) = s^N$. Then
$\TT_{\mathbf{0}}^\mathscr{S} \cong C(\mathscr{S}) \otimes \varinjlim(\TT, \kappa)$. This
isomorphism intertwines the quotient map $q : \TT_{\mathbf{0}}^\mathscr{S} \to
\AA_{\mathbf{0}}^\mathscr{S}$ with the canonical quotient map $\id \otimes \tilde{q} :
C(\mathscr{S}) \otimes \varinjlim(\TT, \kappa) \to C(\mathscr{S}) \otimes
C(\mathscr{S})$. It also intertwines $\alpha$ with $1 \otimes \tilde\alpha$ where
$\tilde\alpha_t(\kappa_{j,\infty}(s)) = e^{it/N^j} \kappa_{j,\infty}(s)$. That is,
$\tilde\alpha$ is equivariant over $\kappa_{j,\infty}$ with an action $\tilde\alpha_j$ on
$\TT$ that is a rescaling of the gauge dynamics studied in \cite{aHLRS}. Theorems
3.1~and~4.3 of \cite{aHLRS} imply that $(\TT, \tilde\alpha_j)$ has a unique $\KMS_\beta$
state for every $\beta \ge 0$ and has no $\KMS_\beta$ states for $\beta < 0$, and that
the $\KMS_0$ state is the only one that factors through $C(\SI)$. So
Proposition~\ref{prop:KMSanddirectlimits} implies that $(\varinjlim(\TT, \kappa),
\tilde\alpha)$ has a unique $\KMS_\beta$ state $\phi_\beta$ for each $\beta \ge 0$ and
has no $\KMS_\beta$ states for $\beta < 0$, and that the $\KMS_0$ state is the only one
that factors through $C(\mathscr{S})$. Hence the map $\psi \mapsto \psi \otimes
\phi_\beta$ determines an affine isomorphism of the state space of $C(\mathscr{S})$ onto
$\KMS_\beta(\TT^{\mathscr{S}}_{\mathbf{0}}, \alpha)$ for each $\beta \ge 0$, there are no
$\KMS_\beta$ states for $\beta < 0$, and the $\KMS_0$ states are the only ones that
factor through $\AA^\mathscr{S}_{\mathbf{0}}$.
\end{remark}

In light of Remark~\ref{rmk:theta=0}, we will from now on consider only those $\theta \in
\Xi_N$ such that $\theta_j \not= 0$ for some $j$. Since $\theta_j \not= 0$ implies
$\theta_{j+1} \not= 0$, and since $\varinjlim((\AA_{\theta_n}, \varphi_n)^\infty_{n=1}) =
\varinjlim((\AA_{\theta_n}, \varphi_n)^\infty_{n=j})$ for any $j$, we may therefore
assume henceforth that $\theta_j \not= 0$ for all $j$.

Our main result is the following.

\begin{thm}\label{thm:main}
Take $N\in\{2, 3, \dots\}$, take $\theta=(\theta_j)_{j=0}^\infty\in\Xi_N$, and take
$\beta\in (0,\infty)$. Suppose that $\theta_j \not= 0$ for all $j$. Then
$\KMS_\beta(\TT_\theta^\mathscr{S},\alpha)$ is isomorphic to the Choquet simplex of Borel
probability measures on the solenoid $\mathscr{S} := \varprojlim(\SI, p_N)$, and there is
an action $\lambda$ of $\mathscr{S}$ on $\TT_\theta^\mathscr{S}$ that induces a free and
transitive action of $\mathscr{S}$ on the extreme boundary of
$\KMS_\beta(\TT_\theta^\mathscr{S},\alpha)$. There is a unique $\KMS_0$-state on
$\TT_\theta^\mathscr{S}$ for $\alpha$, and this is the only KMS state for $\alpha$ that
factors through $\AA^{\mathscr{S}}_\theta$. There are no $\KMS_\beta$ states for $\beta <
0$.
\end{thm}

The first step in proving Theorem~\ref{thm:main} is to combine the results of \cite{AaHR}
on KMS states of local homeomorphism $C^*$-algebras with
Proposition~\ref{prop:KMSanddirectlimits} to characterise the KMS states of
$\TT_\theta^\mathscr{S}$ in terms of subinvariant probability measures on the circle. We
start with some notation.

It is helpful to recall what the results of \cite{AaHR} say in the context of the
topological graphs $E_\gamma$. Recall that $\rho$ denotes the gauge action on
$\TT(E_\gamma)$; we also use $\rho$ for the lift of the gauge action to an action of $\R$
on $\TT(E_\gamma)$. Combining Proposition~4.2 and Theorem~5.1 of \cite{AaHR}, we see that
for each Borel probability measure $\mu$ on $\SI$ that is subinvariant in the sense that
$\mu(R_\gamma(U)) \le e^{\beta} \mu(U)$ for every Borel $U \subseteq \SI$, there is a
$\KMS_\beta$-state $\phi_\mu\in\KMS_\beta(\TT(E_\gamma),\rho)$ satisfying
\begin{equation}\label{eq:phimu}
    \phi_\mu(s_\gamma^a i_\gamma(f) s_\gamma^{*b}) = \delta_{a,b} e^{-a\beta}\int_{\SI} f\,d\mu;
\end{equation}
and moreover, the map $\mu \mapsto \phi_\mu$ is an affine isomorphism of the simplex of
subinvariant Borel probability measures on $\SI$ to $\KMS_\beta(\TT(E_\gamma),\rho)$.

\begin{definition}\label{def:subinvariantmeasures}
Fix $r,s \in [0,\infty)$, and $\gamma\in \SI$. Let $M(\SI)$ denote the set of Borel
probability measures on $\SI$. We define
\[
	\mdis{s}{\gamma}:=\{m\in M(\SI): m(R_{\gamma}(U))\le e^sm(U)\text{ for all Borel $U\subseteq \SI$}\}
\]
and
\begin{equation}\label{eq:Omegasub}
    \mcon{r}:=\{m\in M(\SI): m(R_{t}(U))\le e^{rt}m(U)\text{ for all $t\in[0,\infty)$ and Borel $U\subseteq \SI$}\}.
\end{equation}
\end{definition}

\begin{notation}\label{notation:thersubj}
For the rest of the section we fix $\theta=(\theta_j)_{j=0}^\infty\in\Xi_N$ such that
$\theta_j \not= 0$ for all $j$, and $\beta\in [0,\infty)$. We define
\[
r_j:=\beta/N^j\theta_j\quad\text{ for all $j \in \N$.}
\]
\end{notation}

\begin{thm}\label{thm:mainKMSthm}
Take $N\in\N$ with $N\ge 2$, $\theta=(\theta_j)_{j=0}^\infty\in\Xi_N$, and $\beta\in
[0,\infty)$. Suppose that $\theta_j \not= 0$ for all $j$. Then there is an affine
injection
\[
\omega : \KMS_\beta(\TT_\theta^\mathscr{S},\alpha)\to \varprojlim(\mcon{r_j},m\mapsto m\circ p_N^{-1})
\]
such that
\begin{equation}\label{eq:KMS formula}
\phi\circ \psi_{j,\infty}(s_{\theta_j}^a i_{\theta_j}(f) s_{\theta_j}^{*b})
    = \delta_{a,b} e^{-a\beta/N^j}\int_{\SI} f\,d\omega(\phi)_j
\end{equation}
for each $\phi\in\KMS_\beta(\TT_\theta^\mathscr{S},\alpha)$ and $j\in\N$. If $\beta > 0$,
then $\omega$ is an isomorphism.
\end{thm}

Write $\rho$ for the gauge action on $\TT(E_{\theta_j})$, and $\rho_j$ for the action
$t\mapsto \rho_{e^{it/N^j}}$ of $\R$ on $\TT(E_{\theta_j})$. Since the dynamics $\alpha$
on $\TT_\theta^\mathscr{S}$ is induced by the $\rho_j$,
Proposition~\ref{prop:KMSanddirectlimits} yields an affine isomorphism
\[
\KMS_\beta(\TT_\theta^\mathscr{S},\alpha) \cong
    \varprojlim(\KMS_\beta(\TT(E_{\theta_j}),\rho_j),\phi\mapsto \phi\circ\psi_{j-1}).
\]
For each $j\in\N$ and $t\in\R$ we have $\rho_{j,t} = \rho_{t/N^j}$, so
$\KMS_\beta(\TT(E_{\theta_j}),\rho_j) = \KMS_{\beta/N^j}(\TT(E_{\theta_j}),\rho)$. For
$\beta > 0$, the $\KMS_\beta$ simplex of each $(\TT(E_{\theta_j}),\rho)$ is well
understood by the results of \cite{AaHR} (see the discussion preceding~\eqref{eq:phimu}),
and we use these results to prove the following.

\begin{proposition}\label{prop:fromKMStomeasure}
With the hypotheses of Theorem~\ref{thm:mainKMSthm}, there is an affine injection
\[
\tau:\varprojlim(\KMS_{\beta/N^j}(\TT(E_{\theta_j}),\rho),\phi\mapsto \phi\circ\psi_{j-1})\to \varprojlim(\mdis{\beta/N^j}{\theta_j},m\mapsto m\circ p_N^{-1})
\]
such that $\phi_j=\phi_{\tau((\phi_k)_{k=0}^\infty)_j}$, as defined at~\eqref{eq:phimu},
for all $(\phi_k)_{k=0}^\infty$ and $j\in\N$. If $\beta > 0$ then $\tau$ is an
isomorphism.
\end{proposition}

Throughout the rest of this section we suppress intertwining maps in projective limits.

\begin{proof}[Proof of Proposition~\ref{prop:fromKMStomeasure}]
We first claim that for each $j\in\N$ there is an affine injection $\tau_j$ of
$\KMS_{\beta/N^j}(\TT(E_{\theta_j}),\rho)$ onto $\mdis{\beta/N^j}{\theta_j}$ satisfying
\[
\phi(i_{\theta_j}(f))=\int_\SI f\,d(\tau_j(\phi))
    \quad\text{for all $\phi\in \KMS_{\beta/N^j}(\TT(E_{\theta_j}),\rho)$ and $f\in C(\SI)$,}
\]
and that for $\beta > 0$, this $\tau_j$ is an isomorphism. The statement for $\beta > 0$
follows directly from \cite[Theorem~5.1]{AaHR} (see~\eqref{eq:phimu}.

To prove the claim for $\beta = 0$, recall that the $\KMS_0$ states on
$\TT(E_{\theta_j})$ for $\rho$ are the $\rho$-invariant traces. Let $(i_{\theta_j},
s_{\theta_j})$ be the universal Toeplitz pair for $E_{\theta_j}$. If $\phi$ is a
$\KMS_0$-state, then (with the convention that $s^n := s^{*|n|}$ for $n < 0$),
\begin{align}
\phi(s^n_{\theta_j} i_{\theta_j}(f) s^{*m}_{\theta_j})
    &= \phi(i_{\theta_j}(f) s^{*m}_{\theta_j}s^n_{\theta_j})
    = \phi(i_{\theta_j}(f) s^{n-m}_{\theta_j})\nonumber\\
    &= \int_\SI \phi(\rho_t(i_{\theta_j}(f) s^{n-m}_{\theta_j}))\,d\mu(t)
    = \delta_{m,n} \phi(i_{\theta_j}(f)).\label{eq:phi->m}
\end{align}
So, by the Riesz--Markov--Kakutani representation theorem \cite[Theorem~2.14]{Rudin},
there exists a Borel probability measure $m_\phi$ on $\SI$ such that $\phi(s^n_{\theta_j}
i_{\theta_j}(f) s^{*m}_{\theta_j}) = \int_\SI f(t)\,dm_\phi(t)$. For $f \in C(\SI)_+$, we
have
\begin{align*}
\phi(i_{\theta_j}(f))
    &\ge \phi\big(i_{\theta_j}\big(\sqrt{f}\big) s_{\theta_j} s_{\theta_j}^*i_{\theta_j}\big(\sqrt{f}\big)\big)\\
    &= \phi\big(s_{\theta_j}^*i_{\theta_j}\big(\sqrt{f}\big)i_{\theta_j}\big(\sqrt{f}\big) s_{\theta_j}\big)
    = \phi(s_{\theta_j}^*i_{\theta_j}(f) s_{\theta_j})
    = \phi(i_{\theta_j}(f \circ R_{-\theta_j}).
\end{align*}
Hence $m_\phi(R_{\theta_j}(U)) \le m_\phi(U)$ for all Borel $U$. So $\phi \mapsto m_\phi$
is an affine map from $\KMS_{0}(\TT(E_{\theta_j}),\rho)$ and~\eqref{eq:phi->m} shows that
it is injective. This completes the proof of the claim.

For each $j\in\N$ let $p_j$ be the projection from $\varprojlim
\KMS_{\beta/N^j}(\TT(E_{\theta_j}),\rho)$ to $\KMS_{\beta/N^j}(\TT(E_{\theta_j}),\rho)$,
and $\pi_j$ the projection from $\varprojlim \mdis{\beta/N^j}{\theta_j}$ to
$\mdis{\beta/N^j}{\theta_j}$. Fix an element $(\phi_j)_{j=0}^\infty$ of $\varprojlim
\KMS_{\beta/N^j}(\TT(E_{\theta_j}),\rho)$. For each $k\ge 1$ and $f\in C(\SI)$ we have
\begin{align*}
\int_\SI f\,d(\tau_{k-1}(\phi_{k-1}))
    &= \phi_{k-1}(i_{\theta_{k-1}}(f))
    = \phi_k(\psi_{k-1}(i_{\theta_{k-1}}(f))) \\
    &= \phi_k(i_{\theta_k}(f\circ p_N))
    = \int_\SI (f\circ p_N)\, d(\tau_k(\phi_k))
    = \int_\SI f\, d(\tau_k(\phi_k)\circ p_N^{-1}),
\end{align*}
and hence $\tau_{k-1}(\phi_{k-1})=\tau_k(\phi_k)\circ p_N^{-1}$. It follows that
\[
\tau_{k-1}\circ p_{k-1}((\phi_j)_{j=0}^\infty) = \tau(\phi_{k-1})=\tau_k(\phi_k)\circ p_N^{-1}=\tau_k\circ p_k((\phi_j)_{j=0}^\infty)\circ p_N^{-1},
\]
for each $k\ge 1$. The universal property of $\varprojlim \mdis{\beta/N^j}{\theta_j}$
yields a map
\[
    \tau : \varprojlim\KMS_{\beta/N^j}(\TT(E_{\theta_j}),\rho) \to \varprojlim \mdis{\beta/N^j}{\theta_j},
\]
whose image is $\varprojlim \operatorname{range}(\tau_k)$, satisfying $\pi_k\circ \tau =
\tau_k\circ p_k$ for each $k\in \N$. For $\beta > 0$, we have $\varprojlim
\operatorname{range}(\tau_k) = \varprojlim \mdis{\beta/N^j}{\theta_j}$, and otherwise it
is a compact affine subset, so it now suffices to prove that $\tau$ is an affine
isomorphism onto its range. Since $\tau$ is an injective map from a compact space to a
Hausdorff space, it therefore suffices to show that it is affine and continuous.

Suppose $\sum_{i=1}^q\lambda_i(\phi_j^i)_{j=0}^\infty$ is a convex combination in
$\varprojlim \KMS_{\beta/N^j}(\TT(E_{\theta_j}),\rho)$. For each $k\in\N$ and $f\in
C(\SI)$ we have
\begin{align*}
\int_{\SI} f\, d\Big(\tau_k\Big(\sum_{i=1}^q\lambda_i \phi_k^i\Big)\Big)
    &= \Big(\sum_{i=1}^q\lambda_i \phi_k^i\Big)(i_{\theta_k}(f))
    = \sum_{i=1}^q\lambda_i\phi_k^i(i_{\theta_k}(f))\\
    &= \sum_{i=1}^k\lambda_i\int_{\SI}f\, d(\tau_k(\phi_k^i))
    = \int_{\SI} f\, d\Big(\sum_{i=1}^q\lambda_i\tau_k(\phi_k^i)\Big).
\end{align*}
So $\tau_k\Big(\sum_{i=1}^q\lambda_i \phi_k^i\Big)=\sum_{i=1}^q\lambda_i\tau_k(\phi_k^i)$
by the Riesz--Markov--Kakutani representation theorem \cite[Theorem~2.14]{Rudin}, and it
follows that $\tau$ is affine.

Straightforward arguments using that $\pi_k\circ \tau =  \tau_k\circ p_k$ for each $k\in
\N$, and that each $\tau_k$ is injective, show that $\tau$ is injective. We just need to
show that $\tau$ is continuous. Let $((\phi_j^\lambda)_{j=0}^\infty)_{\lambda\in\Lambda}$
be a net in $\varprojlim\KMS_{\beta/N^j}(\TT(E_{\theta_j}),\rho)$ converging in the
initial topology to $(\phi_j)_{j=0}^\infty$. Then
$p_k(((\phi_j^\lambda)_{j=0}^\infty)_{\lambda\in\Lambda})=(\phi_k^\lambda)_{\lambda\in\Lambda}$
converges weak* to $p_k((\phi_j)_{j=0}^\infty)=\phi_k$ for each $k\in\N$. Since $\tau_k$
is continuous and $\pi_k\circ \tau =  \tau_k\circ p_k$ for each $k\in\N$, we have that
$\pi_k(\tau(((\phi_j^\lambda)_{j=0}^\infty)_{\lambda\in\Lambda}))=\tau_k((\phi_k^\lambda)_{\lambda\in\Lambda})$
converges weak* to $\tau_k(\phi_k)=\pi_k(\tau((\phi_j)_{j=0}^\infty))$. Hence
$\tau(((\phi_j^\lambda)_{j=0}^\infty)_{\lambda\in\Lambda})$ converges in the initial
topology to $\tau((\phi_j)_{j=0}^\infty)$. So $\tau$ is continuous.	
\end{proof}

\begin{remark}\label{rem:fromKMStomeasures}
Fix $\beta > 0$. Let $h$ be the affine isomorphism of
Proposition~\ref{prop:KMSanddirectlimits} and let $\tau$ be the affine isomorphism of
Proposition~\ref{prop:fromKMStomeasure}. Setting $\omega:=\tau\circ h$ gives an affine
isomorphism
\[
\omega : \KMS_\beta(\TT^{\mathscr{S}}_\theta, \alpha) \to \varprojlim\mdis{\beta/N^j}{\theta_j}
\]
satisfying $\phi\circ \psi_{j,\infty}=\phi_{\omega(\phi)_j}$ for each
$\phi\in\KMS_\beta(\TT_\theta^\mathscr{S},\alpha)$ and $j\in\N$. So to prove
Theorem~\ref{thm:mainKMSthm} it now suffices to show that
$\varprojlim\mdis{\beta/N^j}{\theta_j} \cong \varprojlim \mcon{r_j}$.

Fix $(m_j)_{j=0}^\infty\in \varprojlim\mcon{r_j}$. Taking $t=\theta_j$ in the definition of
$\mcon{r_j}$ (see
Definition~\ref{def:subinvariantmeasures}) shows that $\mcon{r_j}\subseteq
\mdis{\beta/N^j}{\theta_j}$. Hence $\varprojlim\mcon{r_j}$ is contained in
$\varprojlim\mdis{\beta/N^j}{\theta_j}$. So we need the reverse containment. We start
with a lemma.
\end{remark}

\begin{lemma}\label{lem:subinvforallt}
Let $m$ be a Borel probability measure on $\SI$, and fix $\gamma \in (0,1)$, $s \in
[0,\infty)$ and $N \in\N$ with $N\ge 2$. Suppose that $m(R_{\gamma/N^k}(U)) \le e^{s/N^k}
m(U)$ for every $k \in \N$ and every Borel set $U \subseteq \SI$. Then
$m\in\mcon{s/\gamma}$.
\end{lemma}
\begin{proof}
We need to show that $m(R_t(U))\le e^{(s/\gamma)t}m(U)$ for all $t\ge 0$ and Borel
$U\subseteq\SI$; or equivalently, that $m(R_{t\gamma}(U))\le e^{st}m(U)$ for all $t\ge 0$
and Borel $U\subseteq\SI$. By the Riesz--Markov--Kakutani representation theorem
\cite[Theorem~2.14]{Rudin}, it suffices to show that
\begin{equation}\label{eq:t-subinv}
    \int_{\SI} f \circ R_{-t\gamma}\,dm \le e^{st} \int_{\SI} f\,dm
\end{equation}
for every $t \ge 0$ and every $f \in C(\SI)_+$. Furthermore, if~\eqref{eq:t-subinv} holds
whenever $0 \le t \le 1$, then for arbitrary $T \in [0,\infty)$, we can
iterate~\eqref{eq:t-subinv} $\lceil T\rceil$ times for $t = \frac{T}{\lceil T\rceil}$ to
obtain~\eqref{eq:t-subinv} for $T$; so it suffices to establish~\eqref{eq:t-subinv} for
$t \in [0,1]$.

Fix $t \in [0,1]$ and $f \in C(\SI)$. Write
\[
    t = \sum^\infty_{i=1} \frac{a_i}{N^i}
\]
where each $a_i \in \{0, \dots, N-1\}$. For each $n \in \N$, let $t_n := \sum^n_{i=1}
\frac{a_i}{N^i}$. So $t_n$ is a monotone increasing sequence in $[0,1]$ converging to
$t$. Since the action $s \mapsto R_s$ of $\R$ on $\SI$ by rotations is uniformly
continuous, we have $f \circ R_{-t_n\gamma} \to f\circ R_{-t\gamma}$ in $\big(C(\SI),
\|\cdot\|_\infty\big)$. Since $m$ is a Borel probability measure, the functional $f
\mapsto \int_{\SI} f\,dm$ is a state, and so
\[
    \int_{\SI} f\circ R_{-t_n\gamma}\,dm \to \int_{\SI} f \circ R_{-t\gamma}\,dm.
\]
So it suffices to show that each $\int_{\SI} f \circ R_{-t_n\gamma} \le e^{st} \int_{\T}
f \,dm$. So fix $n \in \N$. Let $K := \sum^n_{i=1} a_i N^{n-i}$, so that $t > t_n =
\frac{K}{N^n}$. By hypothesis, for every Borel $U$, we have
\[
m(R_{\frac{K\gamma}{N^n}}(U))
    \le e^{\frac{s}{N^n}} m(R_{\frac{(K-1)\gamma}{N^n}}(U))
    \le \cdots
    \le e^{\frac{sK}{N^n}} m(U)
    \le e^{st} m(U),
\]
and it follows that $\int_{\SI} f \circ R_{-t_n\gamma} \le e^{st} \int_{\SI} f \,dm$ as
required.
\end{proof}

\begin{proof}[Proof of Theorem~\ref{thm:mainKMSthm}]
As described in Remark~\ref{rem:fromKMStomeasures}, it suffices to show that
$\varprojlim\mdis{\beta/N^j}{\theta_j}$ is contained in $\varprojlim\mcon{r_j}$. For each
$\gamma\in\SI$ we have $p_N\circ R_{\gamma}=R_{N\gamma}\circ p_N$, which implies
that $p_N^{-1}(R_{N\gamma}(U)) = R_\gamma(p_N^{-1}(U))$ for all Borel $U\subseteq
\SI$. An iterative argument shows that
\begin{equation}\label{eq:iotaandRk}
p_N^{-k}(R_{N^k\gamma}(U))=R_\gamma(p_N^{-k}(U))\quad\text{ for all Borel $U\subseteq \SI$ and $k\in\N$.}
\end{equation}

Fix $(m_j)_{j=0}^\infty\in \varprojlim \mdis{\beta/N^j}{\theta_j}$. Since the connecting
maps in $\varprojlim\mdis{\beta/N^j}{\theta_j}$ and $\varprojlim\mcon{r_j}$ are the same,
it suffices to show that $m_j\in \mcon{r_j}$ for each $j\in\N$. Fix $j\in\N$. For each
$k\in\N$ we have $N^{2k}\theta_{j+k}=\theta_k$ and $m_{j+k}\circ p_N^{-k}=m_j$. These
identities and \eqref{eq:iotaandRk} give
\begin{align*}
m_j(R_{\theta_j/N^k}(U))
    &= m_j(R_{N^k\theta_{j+k}}(U))
	= m_{j+k}(p_N^{-k}(R_{N^k\theta_{j+k}}(U)))\\
	&= m_{j+k}(R_{\theta_{j+k}}(p_N^{-k}(U)))
	\le e^{\beta/N^{j+k}}m_{j+k}(p_N^{-k}(U))
	=e^{\beta/N^{j+k}}m_j(U),
\end{align*}
for every Borel $U\subseteq \SI$. So Lemma~\ref{lem:subinvforallt} with $\gamma=\theta_j$
and $s=\beta/N^j$ gives $m_j\in \mcon{r_j}$.
\end{proof}

\section{Subinvariant measures on \texorpdfstring{$\SI$}{the circle}}\label{sec:submeasures} 		

Throughout the section we fix $r\in [0,\infty)$ and denote Lebesgue measure on $\SI$ by
$\mu$. The main result of this section gives a concrete description of the simplex
$\mcon{r}$ of~\eqref{eq:Omegasub}. Define $W_r : \SI \to [0,\infty)$ by
\[
   	W_r(t)=\Big(\frac{r}{1-e^{-r}}\Big) e^{-rt}.
\]
For each Borel $U\subseteq\SI$, define
\begin{equation}\label{eq:mr def}
   	m_r(U):=\int_U W_r(t)\,dt.
\end{equation}
This defines a Borel probability measure $m_r$ on $\SI$. 	

\begin{thm}\label{thm:mainonsubmeasures}
The space $\mcon{r}$ is the weak$^*$-closed convex hull
$\overline{\operatorname{conv}}\{m_r\circ R_s : 0 \le s < 1\}$. If $r = 0$, then $m_r =
\mu$ and $\mcon{r} = \{\mu\}$.
\end{thm}	

We need a number of results to prove this theorem. 	

\begin{lemma}\label{lem:discrete graph}
Let $m\in\mcon{r}$ and $n \in \N$. For $0 \le j < 2^n$, let $U^n_j = [j/2^n, (j+1)/2^n)
\subseteq \SI$, and let $v^n_j$ be the vector
\begin{equation}\label{eq:vareps}
v^n_j := \frac{1 - e^{-r/2^n}}{1 - e^{-r}}
    \big(e^{-(2^n-j)r/2^n}, \dots, e^{-(2^n-1)r/2^n}, 1, e^{-r/2^n}, e^{-2r/2^n},\dots, e^{-(2^n-(j+1))r/2^n}\big) \in \R^{2^n}.
\end{equation}
Then $\big(m(U^n_0),m(U^n_1), \dots, m(U^n_{2^n-1})\big) \in \operatorname{conv}\{v^n_j :
0\le j < 2^n\}$.
\end{lemma}
\begin{proof}
Let $x=(x_0,x_1,\dots,x_{2^n-1})$ be the vector $\big(m(U^n_0),m(U^n_1), \dots,
m(U^n_{2^n-1})\big)$. For each $0\le j<2^n$ we have
\[
x_j
= m(U^n_{j})
= m(R_{2^{-n}}(U^n_{j+1}))
\le e^{r/2^{n}} m(U^n_{j+1})
= e^{r/2^{n}}x_{j+1},
\]
where addition in indices is modulo $2^n$. Let $C^{2^n}$ denote the graph with vertices
$\Z/2^n\Z$ and edges $\{e_j : j \in \Z/2^n\Z\}$ with $s(e_j) = j$ and $r(e_j) =
j+1\Mod{2^n}$, and let $A_{C^{2^n}}$ denote the adjacency matrix of $C^{2^n}$. Then $x$
satisfies $A_{C^{2^n}} x \le e^{r/2^{n}} x$. So $x$ is subinvariant for $A_{C^{2^n}}$ in
the sense of \cite[Theorem~3.1]{aHLRS}, and is a probability measure because $m$ is. By
\cite[Theorem~3.1(a)]{aHLRS}, there is a vector $y \in [1,\infty)^{\Z/2^n\Z}$ such that
\[
y_j = \sum_{\mu \in (C^{2^n})^*, s(\mu) = j} e^{-r/2^{n}|\mu|}
    = \sum_{k=0}^\infty e^{-kr/2^n}
    = \big(1 - e^{-r/2^{n}}\big)^{-1} \quad\text{ for $j \in \Z/2^n\Z$.}
\]
For $0\le j < 2^n$, define $\epsilon_j \in [0,\infty)^{\Z/2^n\Z}$ by
\[
\epsilon_j(k) = \begin{cases}
1 - e^{-r/2^{n}} &\text{ if $k = j$}\\
0 &\text{ otherwise.}
\end{cases}
\]
We have
\begin{align*}
(I &{}- e^{-r/2^n}A_{C^{2^n}}) v^n_j\\
    &= \frac{1 - e^{-r/2^n}}{1 - e^{-r}}(I - e^{-r/2^n}A_{C^{2^n}})\big(e^{-(2^n-j)r/2^n}, \dots, e^{-(2^n-1)r/2^n}, 1, e^{-r/2^n},\dots,e^{-(2^n-(j+1))r/2^n}\big)\\
    &= \frac{1 - e^{-r/2^n}}{1 - e^{-r}}\Big(\big(e^{-(2^n-j)r/2^n}, \dots,e^{-(2^n-1)r/2^n}, 1, e^{-r/2^n}, \dots,e^{-(2^n-(j+1))r/2^n}\big)\\
    &\hskip8em - e^{-r/2^n}\big(e^{-(2^n-(j+1))r/2^n}, \dots,e^{-(2^n-1)r/2^n}, 1, e^{-r/2^n}, \dots, e^{-(2^n-(j+2))r/2^n}\big)\Big) \\
    &= \frac{1 - e^{-r/2^n}}{1 - e^{-r}}\big(0, \dots, 0, 1 - e^{-r}, 0, \dots, 0\big)\\
    &= (1 - e^{-r/2^n})\big(0, \dots, 0, 1, 0, \dots, 0\big)\\
    &= \epsilon_j.
\end{align*}
So $v^n_j = (I - e^{-r/2^n}A_{C^{2^n}})^{-1}\epsilon_j$. Since the $\epsilon_j$ are the
extreme points of the simplex $\{\epsilon : \epsilon\cdot y = 1\}$, it follows from
\cite[Theorem~3.1(c)]{aHLRS} that the $v_j^n$ are the extreme points of the simplex of
subinvariant probability measures on $\Z/2^n\Z$. Since $x$ is a subinvariant probability
measure, it follows that it is a convex combination of the $v_j^n$.
\end{proof}

We now approximate $m_r$ by convex combinations of restrictions of Lebesgue measure.

\begin{lemma}\label{lem:approx}
For $n \in \N$ and $j \in \Z/2^n\Z$, let $U^n_j = [j/2^n, (j+1)/2^n) \subseteq \SI$, and
let $W_{n,r}$ be the simple function
\[
 W_{n,r} = \sum^{2^n-1}_{j=0} 2^n(v^n_0)_j 1_{U^n_j}.
\]
Let $m_{n,r}$ be the measure $m_{n,r}(U) = \int_U W_{n,r}(t) \,d\mu(t)$ for Borel
$U\subseteq \SI$. Then $\lim_{n \to \infty} \big\|m_r - m_{n,r}\big\|_1 = 0$.
\end{lemma}
\begin{proof}
Fix $n \in \N$ and $0 \le j < 2^{-n}$. Then the average value of $W_r$ over the interval
$U_j^n$ is
\begin{align*}
2^n \int_{U^n_j} W_r(t)\,d\mu(t)
    &= 2^n \int^{(j+1)/2^n}_{j/2^n} \Big(\frac{r}{1-e^{-r}}\Big) e^{-rt}\,d\mu(t)
    = 2^n \Big[\Big(\frac{-1}{1-e^{-r}}\Big) e^{-rt}\Big]^{(j+1)/2^n}_{j/2^n}\\
    &= \Big(\frac{-2^n}{1-e^{-r}}\Big)\Big(e^{-(j+1)r/2^n} - e^{-jr/2^n}\Big)
    = 2^n\Big(\frac{1 - e^{-r/2^n}}{1 - e^{-r}}\Big)e^{-jr/2^n}
    = 2^n(v^n_0)_j,
\end{align*}
the constant value of $W_{n,r}$ on $U_j^n$. The Mean Value Theorem---applied to $\int
W_r(t)\,d\mu(t)$---implies that there exists $c^n_j \in (j/2^n, (j+1)/2^n)$ such that
$W_r(c^n_j) = W_{n,r}(c^n_j)$.

Fix $\epsilon > 0$. The function $W_r$ is uniformly continuous on $[0,1)$, and so there
exists $N \in \N$ such that $|W_r(s) - W_r(t)| < \epsilon$ whenever $s,t \in [0,1)$
satisfy $|s-t| < 2^{-N}$. In particular, for $n \ge N$ and $0 \le j < 2^n$, the point
$c^n_j$ of the preceding paragraph satisfies
\begin{align*}
\sup\{W_r(t) - W_{n,r}(t) :{}& j/2^n \le t < (j+1)/2^n\}\\
    &= \sup\{W_r(t) - W_{n,r}(c^n_j) : j/2^n \le t < (j+1)/2^n\}\\
    &= \sup\{W_r(t) - W_r(c^n_j) : j/2^n \le t < (j+1)/2^n\}
    \le \epsilon.
\end{align*}
So for $n \ge N$,
\begin{align*}
\big\|m_r-m_{n,r}\big\|_1
    &= \int^1_0 |W_r(t) - W_{n,r}(t)|\,d\mu(t)\\
    &= \sum^{2^n-1}_{j=0} \int^{(j+1)/2^n}_{j/2^n} |W_r(t) - W_{n,r}(t)|\,d\mu(t)
    \le \sum^{2^n-1}_{j=0} \int^{(j+1)/2^n}_{j/2^n} \epsilon \,d\mu(t)
    = \epsilon,
\end{align*}
and hence $\lim_{n \to \infty} \big\|m_r - m_{n,r}\big\|_1 = 0$.
\end{proof}

\begin{cor}\label{cor:ccapprox}
Given a sequence $(\lambda^n)^\infty_{n=1}$ of vectors $\lambda^n \in [0,1]^{2^n}$
satisfying $\sum^{2^n-1}_{j=0} \lambda^n_j = 1$ for all $n$, we have
\[
    \lim_{n \to \infty} \Big\|\sum^{2^n-1}_{j=0} \lambda^n_j (m_r\circ R_{j/2^n}) - \sum^{2^n-1}_{j=0} \lambda^n_j (m_{n,r}\circ R_{j/2^n})\Big\|_1 = 0.
\]
\end{cor}
\begin{proof}
The triangle inequality gives
\begin{align*}
\Big\|\sum^{2^n-1}_{j=0} \lambda^n_j (m_r\circ R_{j/2^n}) - \sum^{2^n-1}_{j=0} \lambda^n_j(m_{n,r}\circ R_{j/2^n})\Big\|_1
    &\le \sum^{2^n-1}_{j=0} \lambda^n_j\big\|m_r\circ R_{j/2^n} - m_{n,r}\circ R_{j/2^n}\Big\|_1\\
    &= \big\|m_r - m_{n,r}\big\|_1,
\end{align*}
and so the result follows from Lemma~\ref{lem:approx}.
\end{proof}

\begin{proof}[Proof of Theorem~\ref{thm:mainonsubmeasures}]
We first have to show that each $m_r\circ R_s \in \mcon{r}$. To see that
$m_r\in\mcon{r}$, it suffices to prove that $W_r(R_t(t_0))\le e^{rt}W_r(t_0)$ for all
$t_0\in\SI$ and $t\in [0,\infty)$. Fix such a $t_0$ and $t$, and write $t_0-t=t_1+k$ for
$t_1 \in [0,1)$ and $0 \ge k\in\Z$. Then
\begin{align*}
W_r(R_t(t_0))
    &=W_r(t_1)
    =\Big(\frac{r}{1-e^{-r}}\Big)e^{-rt_1}
    = \Big(\frac{r}{1-e^{-r}}\Big)e^{rk}e^{-r(t_1+k)}\\
	&= \Big(\frac{r}{1-e^{-r}}\Big)e^{rk}e^{-r(t_0-t)}
	=e^{rk}e^{rt}W_r(t_0)
	\le e^{rt}W_r(t_0),
\end{align*}
where the inequality follows because $rk\le 0$. So $m_r\in\mcon{r}$. For $0\le s<1$ and
Borel $U \subseteq \SI$, we have $m_r\circ R_s(R_t(U))=m_r(R_t(R_s(U)))\le e^{rt}m_r\circ
R_s(U)$ for all $t\in [0,\infty)$ and Borel $U\subseteq\SI$, and hence $m_r\circ
R_s\in\mcon{r}$.

Since $\mcon{r}$ is convex and weak$^*$ closed, we have
$\overline{\operatorname{conv}}\{m_r\circ R_s : 0 \le s < 1\}\subseteq\mcon{r}$. For the
reverse containment, fix $m\in\mcon{r}$. For each $n \in \N$ and $0\le j <2^n$ we let
$U_j^n:=[j/2^n,(j+1)/2^n)$, and
\[
    x_n := \big(m(U^n_j)\big)^{2^n-1}_{j=0} \in [0,1]^{2^n}.
\]
By Lemma~\ref{lem:discrete graph} we can express $x_n$ as a convex combination $x_n =
\sum^{2^n-1}_{j=1} \lambda^n_j v^n_j$ of the vectors $\{v^n_0, \dots, v^n_{2^n-1}\}$
described at~\eqref{eq:vareps}. We claim that the measures
\[
    M_n:=\sum^{2^n-1}_{j=0} \lambda^n_j (m_r\circ R_{j/2^n})
\]
converge weak$^*$ to $m$. To see this, fix $f \in C(\SI)_+$. It suffices to prove that
$\int f\,dM_n \to \int f\,dm$. For each $n$, let
\[
    M'_n := \sum^{2^n-1}_{j=0} \lambda^n_j (m_{n,r}\circ R_{j/2^n}).
\]
Corollary~\ref{cor:ccapprox} shows that $\|M_n - M'_n\|_1 \to 0$ and in particular, $\int
f\,dM_n - \int f\,dM'_n \to 0$. So it suffices to prove that
\[\textstyle
    \int f\,dM'_n \to \int f\,dm.
\]

For each $n$, define $f_n : \SI \to \R$ by
\[
    f_n = \sum^{2^n-1}_{j=0} f(j/2^n)1_{U_j^n}.
\]
Since $f$ is uniformly continuous on $\SI$ we have $f_n \to f$ pointwise on $\SI$. Since
$|f|$ and each $|f_n|$ are bounded above by $\|f\|_\infty$, the Dominated Convergence
Theorem implies that $\int f_n\,dm \to \int f\,dm$. So it now suffices to prove that
\[\textstyle
    \Big|\int f_n\,dm - \int f\,dM'_n\Big| \to 0.
\]
Fix $j,k \in \Z/2\Z$. Then $(v_0^n)_{j-k} = (v_j^n)_k$, and hence
\[
\int_{U_k^n}f\, d\big(m_{n,r}\circ R_{j/2^n}\big)
    = 2^n(v_0^n)_{j-k}\int_{U_k^n}f\,d\mu
    = 2^n(v_j^n)_{k}\int_{U_k^n}f\,d\mu.
\]
Hence
\begin{align*}
\Big|\int f_n\,dm - \int f\,dM'_n\Big|
    &= \Big|\sum^{2^n-1}_{i=0} f(i/2^n)m(U_i^n) - \sum_{j=0}^{2^n-1}\lambda_j^n\Big(\sum_{k=0}^{2^n-1}\int_{U_k^n} f\,d\big(m_{n,r}\circ R_{j/2^n}\big)\Big)\Big|\\
    &= \Big|\sum^{2^n-1}_{i=0} f(i/2^n)\Big(\sum_{l=0}^{2^n-1}\lambda_l^nv_l^n\Big)_i - \sum_{j=0}^{2^n-1}\lambda_j^n\sum_{k=0}^{2^n-1}\Big(2^n(v_j^n)_{k}\int_{U_k^n}f\,d\mu\Big)\Big|\\
    &= \Big|\sum^{2^n-1}_{l=0}\lambda_l^n\sum_{i=0}^{2^n-1}\big( f(i/2^n)(v_l^n)_i\big) - \sum_{j=0}^{2^n-1}\lambda_j^n\sum_{k=0}^{2^n-1}\Big(2^n(v_j^n)_{k}\int_{U_k^n}f\,d\mu\Big)\Big|\\
    &= \Big|\sum^{2^n-1}_{j=0} \lambda_j^n\Big(\sum_{i=0}^{2^n-1}\big( f(i/2^n)(v_j^n)_i\big) - \sum_{k=0}^{2^n-1}\Big(2^n(v_j^n)_{k}\int_{U_k^n}f\,d\mu\Big)\Big)\Big|\\
    &= \Big|\sum^{2^n-1}_{j=0} \lambda_j^n\Big(\sum_{i=0}^{2^n-1}\Big( f(i/2^n)(v_j^n)_i - 2^n(v_j^n)_{k}\int_{U_k^n}f\,d\mu\Big)\Big)\Big|.
\end{align*}
Since each $\|v_j^n\|_1=1$ and each $\sum_j \lambda^n_j = 1$, the triangle inequality
gives
\begin{align*}
\Big|\int f_n\,dm - \int f\,dM'_n\Big|
    &\le \sum^{2^n-1}_{j=0} \lambda^n_j \Big|\sum^{2^n-1}_{i=0} \Big(f(i/2^{n}) - 2^n\int_{U_i^n} f\,d\mu\Big) (v_j^n)_i\Big)\Big|\\
    &\le \max_{0 \le j < 2^n} \sum^{2^n-1}_{i=0} (v_j^n)_i \Big|f(i/2^{n}) - 2^n\int_{U_i^n} f\,d\mu\Big|\\
    &\le \max_{0 \le j < 2^n} \Big(\max_{0 \le i < 2^n} \Big|f(i/2^{n}) - 2^n\int_{U_i^n} f\,d\mu\Big|\Big)\\
    &= \max_{0 \le i < 2^n} \Big|f(i/2^{n}) - 2^n\int_{U_i^n} f\,d\mu\Big|.
\end{align*}
Fix $0\le i \le 2^n$. The quantity $2^n\int_{U_i^n} f\,d\mu$ is the average value of $f$
over $U_i^n$. Since $f$ is continuous, the Mean Value Theorem implies that there exists
$c \in U_i^n$ such that $f(c) = 2^n\int_{U_i^n} f\,d\mu$. Hence
\[
\Big|\int f_n\,dm - \int f\,dM'_n\Big|
    \le \max_{0 \le i < 2^n} \sup_{c \in U_i^n} |f(i/2^n) - f(c)|.
\]
Fix $\epsilon > 0$. By uniform continuity of $f$ there exists $N$ such that $|x-y| <
2^{-N} \implies |f(x) - f(y)| < \epsilon$. For $n \ge N$ we have $\sup_{c \in U_i^n}
|f(i/2^n) - f(c)| \le \epsilon$ for all $i$, giving $\Big|\int f_n\,dm - \int
f\,dM'_n\Big| \le \epsilon$. Hence $\Big|\int f_n\,dm - \int f\,d\rho_n\Big| \to 0$. So
$m \in \overline{\operatorname{conv}}\{m_r\circ R_s : 0 \le s < 1\}$, giving $\mcon{r}
\subseteq \overline{\operatorname{conv}}\{m_r\circ R_s : 0 \le s < 1\}$ as required.

For the final statement, observe that
\[
    \mcon{0} = \{m\in M(\SI): m(R_{t}(U))\le m(U)\text{ for all $t\in[0,\infty)$ and Borel $U\subseteq \SI$}\}.
\]
So if $m \in \mcon{0}$, then $m(U) = m(R_{1-t}(R_t(U))) \le m(R_t(U)) \le m(U)$ for all
$U,t$, forcing $m(U) = m(R_t(U))$ for all $U, t$. Uniqueness of the Haar measure $\mu$ on
the compact group $\SI$ therefore gives $m = \mu$. So $\mcon{0} \subseteq \{\mu\}$. The
reverse containment is trivial.
\end{proof}

We can use Theorem~\ref{thm:mainonsubmeasures} to describe the extreme points of
$\mcon{r}$.

\begin{proposition}\label{prop:extpointsofOmega}
The set $\{m_r\circ R_s:0\le s< 1\}$ is the set of extreme points of $\mcon{r}$.
\end{proposition}

The first step in proving Proposition~\ref{prop:extpointsofOmega} will be to show that
$m_r$ itself is an extreme point of $\mcon{r}$. The following lemma will help.

\begin{lemma}\label{lem:mismronsubints}
Let $m\in\mcon{r}$ and $n\in\N$ with $n\ge 1$. If $m([\frac{n-1}{n}, 1)) \le
m_r([\frac{n-1}{n}, 1))$, then $m([\frac{i}{n}, \frac{i+1}{n})) = m_r([\frac{i}{n},
\frac{i+1}{n}))$ for all $0 \le i < n$.
\end{lemma}
\begin{proof} First observe that by definition of $m_r$, we have $m_r(R_t(U)) =
e^{rt}m_r(U)$ whenever $U \cup U - t \subseteq [0,1)$. Using this at the fourth equality,
we note that if $m([\frac{n-1}{n}, 1)) \le m_r([\frac{n-1}{n}, 1))$, then subinvariance forces
\begin{align*}
1 = m(\SI)
    = \sum^{n-1}_{i=0} m\Big(\Big[\frac{i}{n}, \frac{i+1}{n}\Big)\Big)
    &= \sum_{i=0}^{n-1}m\Big(R_{(n-1-i)/n}\Big(\Big[\frac{n-1}{n},1\Big)\Big)\Big)\\
	&\le \sum^{n-1}_{i=0} e^{(n-1-i)r/n} m\Big(\Big[\frac{n-1}{n},1\Big)\Big)\\
	& \le \sum^{n-1}_{i=0} e^{(n-1-i)r/n} m_r\Big(\Big[\frac{n-1}{n},1\Big)\Big)\\
	&= \sum^{n-1}_{i=0} m_r\Big(\Big[\frac{i}{n}, \frac{i+1}{n}\Big)\Big)\\
	&= 1.
\end{align*}
So we have equality throughout. From this we deduce first that
\[
\sum^{n-1}_{i=0} m\Big(\Big[\frac{i}{n}, \frac{i+1}{n}\Big)\Big)
    =\sum^{n-1}_{i=0} e^{(n-1-i)r/n} m\Big(\Big[\frac{n-1}{n},1\Big)\Big).
\]
Since the subinvariance relation forces $m\big(\big[\frac{i}{n}, \frac{i+1}{n}\big)\big)
\le e^{(n-1-i)r/n} m\big(\big[\frac{n-1}{n},1\big)\big)$ for each $i$, we deduce that
$m\big(\big[\frac{i}{n}, \frac{i+1}{n}\big)\big) = e^{(n-1-i)r/n}
m\big(\big[\frac{n-1}{n},1\big)\big)$ for each $i$. Since
\[
\sum^{n-1}_{i=0} e^{(n-1-i)r/n} m\Big(\Big[\frac{n-1}{n},1\Big)\Big)
    = \sum^{n-1}_{i=0} e^{(n-1-i)r/n} m_r\Big(\Big[\frac{n-1}{n},1\Big)\Big),
\]
we also have $m\big(\big[\frac{n-1}{n},1\big)\big) =
m_r\big(\big[\frac{n-1}{n},1\big)\big)$. Hence for each $i$ we have
\[\textstyle
m\Big(\Big[\frac{i}{n}, \frac{i+1}{n}\Big)\Big)
	= e^{(n-1-i)r/n} m\Big(\Big[\frac{n-1}{n},1\Big)\Big)
	= e^{(n-1-i)r/n} m_r\Big(\Big[\frac{n-1}{n},1\Big)\Big)
	= m_r\Big(\Big[\frac{i}{n}, \frac{i+1}{n}\Big)\Big).\qedhere
\]
\end{proof}

\begin{proof}[Proof of Proposition~\ref{prop:extpointsofOmega}]
We first show that $m_r$ is an extreme point of $\mcon{r}$. First suppose $m\in\mcon{r}$
satisfies $m([\frac{n-1}{n},1))\le m_r([\frac{n-1}{n},1))$ for all $n$. We claim that
$m=m_r$. Fix $f\in C(\SI)_+$. For each $n$ define $f_n : \SI \to \R$ by
\[
    f_n = \sum^{n-1}_{i=0} f(i/n)1_{[\frac{i}{n},\frac{i+1}{n})}.
\]
The Dominated Convergence Theorem gives $\int f_n\,dm \to \int f\,dm$. By
Lemma~\ref{lem:mismronsubints},
$m([\frac{i}{n},\frac{i+1}{n}))=m_r([\frac{i}{n},\frac{i+1}{n}))$ for all $n\ge 1$ and
$0\le i<n$. Hence the Dominated Convergence Theorem gives $\int f_n\,dm = \int
f_n\,dm_r\to \int f\, dm_r$. It follows that $m=m_r$.

Now suppose that $m_1, m_2\in\mcon{r}$, $t \in (0,1)$ and that one of $m_1$ and $m_2$ is
not equal to $m_r$; say $m_1 \not= m_r$. The above claim yields $n$ such that
$m_1([\frac{n-1}{n},1)) > m_r([\frac{n-1}{n},1))$. So
\[
(t m_1 + (1-t) m_2)\Big(\Big[\frac{n-1}{n},1\Big)\Big)
    > (tm_r + (1-t)m_2)\Big(\Big[\frac{n-1}{n},1\Big)\Big)
    \ge m_r\Big(\Big[\frac{n-1}{n},1\Big)\Big),
\]
and hence $tm_1 + (1-t)m_2 \not= m_r$. So $m_r$ cannot be expressed as a nontrivial
convex combination of subinvariant probability measures, and hence is an extreme point of
$\mcon{r}$.

For $s \in \SI$, the map $m \mapsto m \circ R_s$ is an affine homeomorphism of
$\mcon{r}$, so each $m \circ R_s$ is an extreme point of $\mcon{r}$. This gives $\{m_r
\circ R_s : s \in \SI\} \subseteq \partial\mcon{r}$.

For the reverse containment, observe that the space $\mcon{r}$ of all subinvariant
probability measures on $\SI$ is a weak$^*$-compact convex subset of the Banach space of
all signed Borel measures on $\SI$. The map $s \mapsto m_r\circ R_s$ is a homeomorphism
of $\SI$ onto $Z := \{m_r\circ R_s : s \in \SI\}$. So $Z$ is compact and in particular
closed. Since $\mcon{r}$ is the closed convex hull of $Z$ it follows from
\cite[Proposition~1.5]{Phelps:Choquet} that the set of extreme points of $\mcon{r}$ is
contained in the closure of $Z$ and therefore in $Z$ itself.
\end{proof}

\section{Proof of the main theorem}\label{sec:mainproof}

We are now almost ready to prove Theorem~\ref{thm:main}. We saw in
Theorem~\ref{thm:mainKMSthm} that the $\KMS_\beta$ simplex of $\TT_\theta^\mathscr{S}$ is
affine isomorphic to the projective limit of the $\mcon{r_j}$ under the maps induced by
the covering maps $p_N : \SI \to \SI$. So we now show that these induced maps carry
extreme points to extreme points.

\begin{lemma}\label{lem:extreme covering}
Let $N\in\N$ with $N\ge 2$, $\theta=(\theta_j)_{j=0}^\infty\in\Xi_N$, and $\beta\in
(0,\infty)$. Suppose that $\theta_j \not= 0$ for all $j$. For each $j \in \N$, let $r_j
:= \frac{\beta}{N^j\theta_j}$, and let $m_{r_j}$ be the subinvariant measure on $\SI$
defined by~\eqref{eq:mr def}. For each $s \in [0,1)$, we have $m_{r_{j+1}} \circ R_s
\circ p_N^{-1} = m_{r_j} \circ R_{Ns}$.
\end{lemma}
\begin{proof}
We first establish the result with $s = 0$. Fix $0 \le a < b \le 1$. It suffices to prove
that $m_{r_{j+1}}\circ p_{N}^{-1}\big((a,b)\big) = m_{r_j}\big((a,b)\big)$. We have
\begin{equation}\label{eq:easy side}
m_{r_j}\big((a,b)\big)
    = \int^b_a W_{r_j}(t)\,dt
    = \frac{r_j}{1 - e^{-r_j}} \int^b_a e^{-r_j t}\,dt
    = \frac{-1}{1 - e^{-r_j}}\big(e^{-r_jb} - e^{-r_j a}\big).
\end{equation}
We also have
\begin{equation}\label{eq:hard side1}
m_{r_{j+1}} \circ p_N^{-1}\big((a,b)\big)
    = \sum^N_{i=0} m_{r_{j+1}}\bigg(\Big(\frac{a+i}{N}, \frac{b+i}{N}\Big)\bigg)
    = \sum^N_{i=0} \int_{\frac{a+i}{N}}^{\frac{b+i}{N}} W_{r_{j+1}}(t)\,dt.
\end{equation}
Since
\[
\int W_{r_{j+1}}(t)\,dt
    = \int \Big(\frac{r_{j+1}}{1-e^{-r_{j+1}}}\Big) e^{-r_{j+1}t} \,dt
    = \frac{-1}{1 - e^{-r_{j+1}}} e^{-r_{j+1} t},
\]
Equation~\eqref{eq:hard side1} gives
\begin{align}
m_{r_{j+1}} \circ p_N^{-1}\big((a,b)\big)
    &= \frac{-1}{1 - e^{-r_{j+1}}} \sum^N_{i=0} \Big[e^{-r_{j+1} t}\Big]_{\frac{a+i}{N}}^{\frac{b+i}{N}}\nonumber\\
    &= \frac{-1}{1 - e^{-r_{j+1}}} \sum^N_{i=0} e^{-\frac{i}{N}r_{j+1}}\big(e^{-\frac{b}{N} r_{j+1}} - e^{-\frac{a}{N} r_{j+1}}\big)\nonumber\\
    &= \frac{-1}{1 - e^{-r_{j+1}}} \frac{1 - e^{-r_{j+1}}}{1 - e^{-\frac{r_{j+1}}{N}}} \big(e^{-\frac{b}{N} r_{j+1}} - e^{-\frac{a}{N} r_{j+1}}\big)\nonumber\\
    &= \frac{-1}{1 - e^{-\frac{r_{j+1}}{N}}} \big(e^{-\frac{b}{N} r_{j+1}} - e^{-\frac{a}{N} r_{j+1}}\big).\label{eq:hard side}
\end{align}
Since $N^2\theta_{j+1} = \theta_j$, we have
\[
\frac{r_{j+1}}{N} = \frac{\beta/(N^{j+1}\theta_{j+1})}{N} = \beta/(N^j \cdot N^2\theta_{j+1}) = \beta/N^j\theta_j = r_j,
\]
and so~\eqref{eq:hard side} is precisely~\eqref{eq:easy side}.

Now for $s \not= 0$, observe that $p_N \circ R_s = R_{Ns} \circ p_N$ so that
$R_s(p_N^{-1}(U)) = p_N^{-1}(R_{Ns}(U))$ for all $U \subseteq \SI$. Hence
\[
m_{r_{j+1}} \circ R_s \circ p_N^{-1} = m_{r_{j+1}} \circ p_N^{-1} \circ R_{Ns} = m_{r_j} \circ R_{Ns}.\qedhere
\]
\end{proof}

We now describe the extreme points of the space $\varprojlim(\mcon{r_j}, m \mapsto m
\circ p_N^{-1})$. Given a Borel map $\psi : X \to Y$, we write $\psi_* : M(X) \to M(Y)$
for the induced map $\psi_*(m)(U) = m(\psi^{-1}(U))$.

\begin{lemma}\label{lem:extreme pts}
Take $N\in\{2, 3, \dots\}$, fix $\theta=(\theta_j)_{j=0}^\infty\in\Xi_N$, and fix
$\beta\in (0,\infty)$. Suppose that $\theta_j \not= 0$ for all $j$. For each $j \in \N$,
let $r_j := \frac{\beta}{N^j\theta_j}$, and let $m_{r_j}$ be the subinvariant measure on
$\SI$ defined by~\eqref{eq:mr def}. The map $\pi : (s_j)^\infty_{j=1} \mapsto (m_{r_j}
\circ R_{s_j})^\infty_{j=1}$ is a homeomorphism of $\varprojlim(\SI, p_N)$ onto the
set of extreme points of $\varprojlim(\mcon{r_j}, (p_N)_*)$.
\end{lemma}
\begin{proof}
Since the $\mcon{r_j}$ are compact convex sets and $(p_N)_*$ is affine and
continuous, the projective limit $\varprojlim \mcon{r_j}$ is a compact convex set. The
map $\pi$ is continuous, so its range is compact and hence closed. So to see that the
image of $\pi$ contains all of the extreme points of $\varprojlim \mcon{r_j}$, it
suffices by \cite[Proposition~1.5]{Phelps:Choquet} to show that $\varprojlim \mcon{r_j}$
is contained in the closed convex hull of the $\pi\big((s_j)^\infty_{j=1}\big)$.

For this, fix a point $(m_j)^\infty_{j=1} \in \varprojlim \mcon{r_j}$. Take an open
neighbourhood $U$ of $(m_j)$. By definition of the projective-limit topology, there exist
$k \in \N$ and $U_k \subseteq \mcon{r_k}$ open such that the cylinder set $Z(U_k)$ satisfies $(m_j)^\infty_{j=1} \in Z(U_k)
\subseteq U$. By Theorem~\ref{thm:mainonsubmeasures}, there exist $t_1, \dots, t_L \in
[0,1]$ with $\sum t_l = 1$ such that
\[\textstyle
\sum^L_{l=1} t_l (m_{r_k} \circ R_{s_l}) \in U_k.
\]
Now for each $j \in \N$, define $m'_j := \sum^L_{l=1} t_l (m_{r_l} \circ R_{N^{j-l} s_l})$.
Lemma~\ref{lem:extreme covering} shows that for $j \le j' \in \N$ we have $m'_j =
(p_N)_*^{j'-j}(m'_{j'})$, and so $(m'_j)^\infty_{j=1} \in \varprojlim \mcon{r_j}$. For $l
\le L$, we have $(m_{r_l} \circ R_{N^{j-l} s_l})^\infty_{j=1} = \pi\big((N^{j-l}
s_l)^\infty_{j=1}\big)$, and so
\[
    (m'_j)^\infty_{j=1} \in \operatorname{conv}\pi\big(\varprojlim(\SI, p_N)\big) \cap U.
\]
That is, $\varprojlim \mcon{r_j} \subseteq
\overline{\operatorname{conv}}\big(\pi\big(\varprojlim \SI\big)\big)$. So the range of
$\pi$ contains all the extreme points of $\varprojlim(\mcon{r_j}, (p_N)_*)$.

For the reverse containment, it suffices to show that each
$\pi\big((s_j)^\infty_{j=1}\big)$ is an extreme point of $\varprojlim \mcon{r_j}$. For
this, suppose that $t \in (0,1)$ and $m', m'' \in \varprojlim \mcon{r_j}$ satisfy
\[
\pi\big((s_j)^\infty_{j=1}\big) = tm' + (1-t)m''.
\]
For each $j$,
\[
m_{r_j} \circ R_{s_j}
    = \pi\big((s_j)^\infty_{j=1}\big)_j
    = (tm' + (1-t)m'')_j
    = tm'_j + (1-t)m''_j.
\]
Proposition~\ref{prop:extpointsofOmega} shows that each $m_{r_j} \circ R_{s_j}$ is an
extreme point of $\mcon{r_j}$, forcing $m'_j = m''_j = m_{r_j} \circ R_{s_j}$. So $m' =
m'' = \pi\big((s_j)^\infty_{j=1}\big)$.

Finally, $\pi$ is a homeomorphism onto its range because it is a continuous injection
from a compact space to a Hausdorff space.
\end{proof}

The final ingredient needed for the proof of Theorem~\ref{thm:main} is a suitable action
$\lambda$ of $\mathscr{S}$ on $\TT^{\mathscr{S}}_\theta$.

\begin{lemma}\label{lem:S-action}
There is an action $\lambda$ of $\mathscr{S} = \varprojlim(\SI, p_N)$ on
$\TT^{\mathscr{S}}_\theta$ such that
\[
\lambda_{(s_j)^\infty_{j=1}}\big(\psi_{j,\infty}(s_{\theta_j}^a i_{\theta_j}(f) s_{\theta_j}^{*b})\big)
    = \psi_{j,\infty}\big(s_{\theta_j}^a i_{\theta_j}(f \circ R_{s_j}) s_{\theta_j}^{*b}\big)
\]
for all $j,a,b \ge 0$ and $f \in C(\SI)$.
\end{lemma}
\begin{proof}
For each $j \in \N$, and each $t \in \SI$, there is an automorphism of the topological
graph $E_{\theta_j}$ given by $s \mapsto s+t$ for $s \in E_{\theta_j}^0 = \SI$, and $s
\mapsto s+t$ for $s \in E_{\theta_j}^1 = \SI$. This automorphism induces an automorphism
$\lambda_{j,t}$ of $\TT(E_{\theta_j})$ such that $\lambda_{j,t}(s_{\theta_j}^a
i_{\theta_j}(f) s_{\theta_j}^{*b}) = s_{\theta_j}^a i_{\theta_j}(f \circ R_t)
s_{\theta_j}^{*b}$ for all $j,a,b \ge 0$ and $f \in C(\SI)$.

Since $\lambda_{j,t}(s_{\theta_j}) = s_{\theta_j}$ and $\lambda_{j,t}(i_{\theta_j}(f)) =
i_{\theta_j}(f \circ R_t)$ for all $f \in C(\SI)$, a routine calculation shows that for
$(s_j)^\infty_{j=1} \in \mathscr{S}$, we have $\psi_j \circ \lambda_{j, s_j} =
\lambda_{j+1, s_{j+1}} \circ \psi_j$, and so the universal property of the direct limit
yields the desired action $\lambda$ of $\mathscr{S}$ on $\varinjlim(\TT(E_{\theta_j}),
\psi_j) = \TT^{\mathscr{S}}_\theta$.
\end{proof}

\begin{proof}[Proof of Theorem~\ref{thm:main}]
Theorem~\ref{thm:mainKMSthm} yields an affine isomorphism
\[
\omega : \KMS_\beta(\TT_\theta^\mathscr{S},\alpha)\to \varprojlim(\mcon{r_j}, (p_N)_*).
\]
Lemma~\ref{lem:extreme pts} shows that the space of extreme points of $\varprojlim
\mcon{r_j}$ is homeomorphic to the solenoid $\varprojlim \SI$, so the extreme boundary of
$\KMS_\beta(\TT_\theta^\mathscr{S},\alpha)$ is homeomorphic to $\varprojlim \SI$. As
discussed on pages 141~and~138 of \cite{TakesakiWinnink:CMP1973}, the set of KMS states
for a given dynamics on a unital $C^*$-algebra at given inverse temperature $\beta$ is a
Choquet simplex. So $\KMS_\beta(\TT_\theta^\mathscr{S},\alpha)$ is a Choquet simplex, and
therefore affine isomorphic to the simplex of Borel probability measures on its extreme
boundary.

We claim that the action $\lambda$ of Lemma~\ref{lem:S-action} induces a free and
transitive action of $\mathscr{S}$ on the extreme boundary of the KMS$_\beta$-simplex.
The formula~\eqref{eq:KMS formula} shows that for $l \in \N$, we have
\[
\omega(\phi \circ \lambda_{(s_j)^\infty_{j=1}})_l
    = \omega(\phi)_l \circ R_{s_l}.
\]
That is, for $(m_j)^\infty_{j=1} \in \varprojlim(\mcon{r_j})$, we have
$\omega^{-1}((m_j)^\infty_{j=1}) \circ \lambda_{(s_j)^\infty_{j=1}} = \omega^{-1}((m_j
\circ R_{s_j})^\infty_{j=1})$. In particular, if $\pi : \varprojlim \SI \to \varprojlim
\mcon{r_j}$ is the map of Lemma~\ref{lem:extreme pts}, then
\[
\omega^{-1}(\pi((t_j)^\infty_{j=1})) \circ \lambda_{(s_j)^\infty_{j=1}}
    = \omega^{-1}(\pi((t_j - s_j)^\infty_{j=1})).
\]
That is, the homeomorphism $\omega^{-1} \circ \pi$ of $\mathscr{S}$ onto the extreme
boundary of $\KMS_\beta(\TT^{\mathscr{S}}_\theta, \alpha)$ intertwines $\lambda$ with the
action of $\mathscr{S}$ on itself by translation, which is free and transitive.

Now suppose that $\beta = 0$. Then each $\mcon{r_j} = \mcon{0} = \{\mu\}$, and so
Theorem~\ref{thm:mainKMSthm} gives an affine injection of
$\KMS_0(\TT_\theta^\mathscr{S},\alpha)$ into the 1-point space $\varprojlim(\{\mu\},
\operatorname{id})$. So there is at most one $\KMS_0$-state. That there is one follows
from a standard argument: Choose $\beta_n \in (0,\infty)$ converging to $0$. For each
$n$, fix $\phi_n \in \KMS_{\beta_n}(\TT_\theta^\mathscr{S}, \alpha)$.
Weak$^*$-compactness of the state space ensures that the $\phi_n$ have a convergent
subsequence. Its limit is a $\KMS_0$-state by \cite[Proposition 5.3.23]{BRII}.

It remains to show that the $\KMS_0$ state is the only one that factors through
$\AA^{\mathscr{S}}_\theta$, and that there are no $\KMS_\beta$ states for $\beta < 0$.
For any $\beta$, if $\phi$ is a $\KMS_\beta$ state of $\TT^{\mathscr{S}}_\theta$, then in
particular,
\begin{equation}\label{eq:phi on gap}
\phi(\psi_{1,\infty}(s_{\theta_1} s^*_{\theta_1}))
    = \phi(\psi_{1,\infty}(s^*_{\theta_1})\alpha_{i\beta}(\psi_{1,\infty}(s_{\theta_1})))
    = e^{-\beta}\phi(\psi_{1,\infty}(s^*_{\theta_1}s_{\theta_1}))
    = e^{-\beta}\phi(1_{\TT_\theta^\mathscr{S}}),
\end{equation}
and since $\phi$ is a state, we deduce that $\phi(1_{\TT_\theta^\mathscr{S}} -
\psi_{1,\infty}(s_{\theta_1} s^*_{\theta_1})) = 1 - e^{-\beta}$. Since $s_{\theta_1}$ is
an isometry, we have $1_{\TT_\theta^\mathscr{S}} - \psi_{1,\infty}(s_{\theta_1}
s^*_{\theta_1}) \ge 0$ forcing $1 - e^{-\beta} \ge 0$ and hence $\beta \ge 0$. So there
are no $\KMS_\beta$ states for $\beta < 0$.

If $\beta > 0$, then~\eqref{eq:phi on gap} shows that $\phi(1_{\TT_\theta^\mathscr{S}} -
\psi_{1,\infty}(s_{\theta_1} s^*_{\theta_1})) > 0$, whereas the image of
$1_{\TT_\theta^\mathscr{S}} - \psi_{1,\infty}(s_{\theta_1} s^*_{\theta_1})$ in
$\AA^{\mathscr{S}}_\theta$ is equal to zero. Hence $\phi$ does not factor through
$\AA^{\mathscr{S}}_\theta$.

It remains to prove that if $\phi$ is a $\KMS_0$ state, then $\phi$ factors through
$\AA^{\mathscr{S}}_\theta$. Equation~\ref{eq:phi on gap} implies that
$\phi(1_{\TT_\theta^\mathscr{S}} - \psi_{1,\infty}(s_{\theta_1} s^*_{\theta_1})) = 0$.
The projection $1_{\TT_\theta^\mathscr{S}} - \psi_{1,\infty}(s_{\theta_1}
s^*_{\theta_1})$ is fixed by $\alpha$, and Lemma~\ref{lem:kernel} implies that it
generates the kernel of the quotient map $q : \TT^{\mathscr{S}}_\theta \to
\AA^{\mathscr{S}}_\theta$. So \cite[Lemma~2.2]{aHLRS} implies that $\phi$ factors through
$\AA^{\mathscr{S}}_\theta$.
\end{proof}

\end{document}